\documentclass[11pt, letterpaper]{article}    	
 \usepackage[margin=1in]{geometry}
\usepackage{setspace}
\usepackage{graphicx}	
\usepackage{subcaption}
\usepackage{float}

\usepackage[draft]{todonotes}   

\usepackage{enumerate}			
\usepackage{enumitem}
\setlist[enumerate]{itemsep=-0.5ex,partopsep=1ex,parsep=1ex}
								
\usepackage{amssymb}
\usepackage{amsmath, amsthm, amssymb}

\newtheorem{thm}{Theorem}[section]
\newtheorem{conj}[thm]{Conjecture}
\newtheorem{prop}[thm]{Proposition}
\newtheorem{claim}{Claim}
\newtheorem{lemma}[thm]{Lemma}
\newtheorem{cor}[thm]{Corollary}

\theoremstyle{definition}
\newtheorem{defi}[thm]{Definition}

\usepackage{tikz}
\newcounter{casenum}

\newcommand*{\rom}[1]{\expandafter{\romannumeral #1\relax}}

\usepackage{authblk}
\makeatletter
\def\@cite#1#2{{\normalfont[{\bfseries#1\if@tempswa , #2\fi}]}}
\makeatother

\title{The typical structure of Gallai colorings and their extremal graphs}

\author{ J\'ozsef Balogh\thanks{Department of Mathematical Sciences, University of Illinois at Urbana-Champaign, IL, USA, and Moscow Institute of Physics and Technology, 9 Institutskiy per., Dolgoprodny, Moscow Region, 141701, Russian Federation. Research of the first author is partially supported by NSF Grants DMS-1500121, DMS-1764123 and Arnold O. Beckman Research Award (UIUC) Campus Research Board 18132 and the Langan Scholar Fund (UIUC).} \ \ \ \ \ \ \ Lina Li\thanks{Department of Mathematics, University of Illinois at Urbana-Champaign, Urbana, Illinois 61801, USA. Email: linali2@illinois.edu.}}

\begin{document}
\maketitle

\begin{abstract}
An edge coloring of a graph $G$ is a \textit{Gallai coloring} if it contains no rainbow triangle. 
We show that the number of Gallai $r$-colorings of $K_n$ is $\left(\binom{r}{2}+o(1)\right)2^{\binom{n}{2}}$. This result indicates that almost all Gallai $r$-colorings of $K_n$ use only 2 colors. 
We also study the extremal behavior of Gallai $r$-colorings among all $n$-vertex graphs.
We prove that the complete graph $K_n$ admits the largest number of Gallai $3$-colorings among all $n$-vertex graphs when $n$ is sufficiently large, while for $r\geq 4$, it is the complete bipartite graph $K_{\lfloor n/2 \rfloor, \lceil n/2  \rceil}$. 
Our main approach is based on the hypergraph container method, developed independently by Balogh, Morris and Samotij as well as by Saxton and Thomason, together with some stability results.
\end{abstract}

\section{Introduction}
An edge coloring of a graph $G$ is a \textit{Gallai coloring} if it contains no rainbow triangle, that is, no triangle is colored with three distinct colors. The term \textit{Gallai coloring} was first introduced by Gy\'{a}rf\'{a}s and Simonyi~\cite{GS}, but this concept had already occurred in an important result of Gallai~\cite{G} on comparability graphs, which can be reformulated in terms of Gallai colorings. It also turns out that Gallai colorings are relevant to generalizations of the perfect graph theorem~\cite{CEL}, and some applications in information theory~\cite{KS}. There are a variety of papers which consider structural and Ramsey-type problems on Gallai colorings, see, e.g., \cite{FGP, GSa, GSSS, GS, W}.

Two important themes in extremal combinatorics are to enumerate discrete structures that have certain  properties and describe their typical properties.
In this paper, we shall be concerned with Gallai colorings from such an extremal perspective. 
 
\subsection{Gallai colorings of complete graphs}
For an integer $r\geq 3$, an \textit{$r$-coloring} is an edge coloring that uses at most $r$ colors.
By choosing two of the $r$ colors and coloring the edges of $K_n$ arbitrarily with these two colors, one can easily obtain that the number of Gallai $r$-colorings of $K_n$ is at least
\begin{equation}\label{2color:bound}
\binom{r}{2}\left(2^{\binom{n}{2}}-2\right)+r = \binom{r}{2}2^{\binom{n}{2}}-r(r-2).
\end{equation}
If we further consider all Gallai $r$-colorings of $K_n$ using exactly 3 colors, red, green, and blue, in which the red color is used only once, the number of them is exactly
\[
\binom{n}{2}\left(2^{\binom{n}{2} - (n-1)} - 2\right).
\]
Combining with (\ref{2color:bound}), for $n$ sufficiently large, a trivial lower bound for the number of Gallai $r$-colorings of $K_n$ is
\begin{equation}\label{tri:lbound}
\left(\binom{r}{2}+2^{-n}\right)2^{\binom{n}{2}}.
\end{equation}
Motivated by a question of Erd\H{o}s and Rothschild~\cite{E} and the resolution by Alon, Balogh, Keevash and Sudakov~\cite{ABKS}, Benevides, Hoppen and Sampaio~\cite{BHS} studied the general problem of counting the number of edge colorings of a graph that avoid a subgraph colored with a given pattern. In particular, they proved that the number of Gallai $3$-colorings of $K_n$ is at most $\frac32(n-1)!\cdot 2^{\binom{n-1}{2}}$. At the same time, Falgas-Ravry, O'Connell, and Uzzell~\cite{FOSU} provided a weaker upper bound of the form $2^{(1 + o(1))\binom{n}{2}}$, which is a consequence of the multi-color container theory. Very recently, Bastos, Benevides, Mota and Sau~\cite{BBMS} improved the upper bound to $7(n+1)2^{\binom{n}{2}}$. 
Note that the gap between the best upper bound and the trivial lower bound is a linear factor. We show that the lower bound is indeed closer to the truth, and this actually applies for any integer $r$. Our first main result is as follows.
\begin{thm}\label{rcomplete}
For every integer $r\geq3$, there exists $n_0$ such that for all $n>n_0$, the number of Gallai $r$-colorings of the complete graph $K_n$ is at most 
\[\left(\binom{r}{2}+2^{-\frac{n}{4\log^2 n}}\right)2^{\binom{n}{2}}.\]
\end{thm}

Given a class of graphs $\mathcal{A}$, we denote $\mathcal{A}_n$ the set of graphs in $\mathcal{A}$ of order $n$. We say that \textit{almost all graphs in $\mathcal{A}$ has property $\mathcal{B}$} if
$$\lim_{n\rightarrow \infty}\frac{|\{G\in \mathcal{A}_n: G \text{ has property } \mathcal{B}\}|}{|\mathcal{A}_n|}=1.$$
Recall that the number of Gallai $r$-colorings with at most 2 colors is $\binom{r}{2}2^{\binom{n}{2}}-r(r-2)$. Then the description of the typical structure of Gallai $r$-colorings immediately follows from Theorem~\ref{rcomplete}.
\begin{cor}
For every integer $r\geq3$, almost all Gallai $r$-colorings of the complete graph are 2-colorings.
\end{cor}

\subsection{The extremal graphs of Gallai colorings}
There have been considerable advances in edge coloring problems whose origin can be traced back to a question of Erd\H{o}s and Rothschild~\cite{E}, who asked which $n$-vertex graph admits the largest number of $r$-colorings avoiding a copy of $F$ with a prescribed colored pattern, where $r$ is a positive integer and $F$ is a fixed graph. In particular, the study for the extremal graph of Gallai colorings, that is the case when $F$ is a triangle with rainbow pattern, has received attention recently. 
A graph $G$ on $n$ vertices is \textit{Gallai $r$-extremal} if the number of Gallai $r$-colorings of $G$ is largest over all graphs on $n$ vertices. 
For $r\geq 5$, the Gallai $r$-extremal graph has been determined by Hoppen, Lefmann and Odermann~\cite{hoppen2015rainbow, hoppen2017graphs, hoppen2017rainbow}.

\begin{thm}\cite{hoppen2017graphs}
For all $r\geq 10$ and $n\geq 5$, the only Gallai $r$-extremal graph of order $n$ is the complete bipartite graph $K_{\lfloor n/2 \rfloor, \lceil n/2  \rceil}$. 
\end{thm}

\begin{thm}\cite{hoppen2017graphs}\label{5extremal}
For all $r\geq 5$, there exists $n_0$ such that for all $n>n_0$, the only Gallai $r$-extremal graph of order $n$ is the complete bipartite graph $K_{\lfloor n/2 \rfloor, \lceil n/2  \rceil}$. 
\end{thm}

For the cases $r\in \{3, 4\}$, several approximate results were given.
\begin{thm}\label{BHS}\cite{BHS}
There exists $n_0$ such that the following hold for all $n>n_0$.
\begin{enumerate}[label=(\roman*)]
\item For all $\delta>0$, if $G$ is a graph of order $n$, then the number of Gallai $3$-colorings of $G$ is at most $2^{(1 + \delta)n^2/2}.$
\item For all $\xi>0$, if $G$ is a graph of order $n$, and $e(G)\leq (1 - \xi)\binom{n}{2}$, then the number of Gallai $3$-colorings of $G$ is at most $2^{\binom{n}{2}}.$
\end{enumerate}
\end{thm}

We remark that the part (i) of Theorem~\ref{BHS} was also proved in~\cite{hoppen2017graphs}, and the authors further provided an upper bound for $r=4$.

\begin{thm}\cite{hoppen2017graphs}
There exists $n_0$ such that the following hold for all $n>n_0$. For all $\delta>0$, if $G$ is a graph of order $n$, then the number of Gallai $4$-colorings of $G$ is at most $4^{(1 + \delta)n^2/4}.$
\end{thm}

The above theorems show that for $r\in\{3, 4\}$, the complete graph $K_n$ is not far from being Gallai $r$-extremal, while for $r=4$, the complete bipartite graph $K_{\lfloor n/2 \rfloor, \lceil n/2  \rceil}$ is also close to be Gallai $r$-extremal.
Benevides, Hoppen and Sampaio~\cite{BHS} made the following conjecture.
\begin{conj}\label{3extremal}~\cite{BHS} 
The only Gallai $3$-extremal graph of order $n$ is the complete graph $K_n$. 
\end{conj}
For the case $r=4$, Hoppen, Lefmann and Odermann~\cite{hoppen2017graphs} believed that $K_{\lfloor n/2 \rfloor, \lceil n/2  \rceil}$ should be the extremal graph.
\begin{conj}\label{4extremal}~\cite{hoppen2017graphs}
The only Gallai $4$-extremal graph of order $n$ is the complete bipartite graph $K_{\lfloor n/2 \rfloor, \lceil n/2  \rceil}$.
\end{conj}

Using a similar technique as in Theorem~\ref{rcomplete}, we prove an analogous result for dense non-complete graphs when $r=3$.
\begin{thm}\label{3graph}
For $0<\xi \leq \frac{1}{64}$, there exists $n_0$ such that for all $n>n_0$ the following holds. If $G$ is a graph of order $n$, and $e(G)\geq (1-\xi)\binom{n}{2}$, then the number of Gallai 3-colorings of $G$ is at most 
\[
3\cdot 2^{e(G)} + 2^{-\frac{n}{4\log^2 n}}2^{\binom{n}{2}}.
\]
\end{thm}

Together with Theorem~\ref{BHS} and the lower bound~(\ref{tri:lbound}), Theorem~\ref{3graph} solves Conjecture~\ref{3extremal} for sufficiently large $n$.
\begin{thm}
There exists $n_0$ such that for all $n>n_0$, among all graphs of order $n$, the complete graph $K_n$ is the unique Gallai $3$-extremal graph. 
\end{thm}

Our third contribution is the following theorem.

\begin{thm}\label{rgraph}
For $n, r\in \mathbb{N}$ with $r\geq 4$, there exists $n_0$ such that for all $n>n_0$ the following holds. If $G$ is a graph of order $n$, and $e(G)>\lfloor n^2/4 \rfloor$, then the number of Gallai $r$-colorings of $G$ is less than $r^{\lfloor n^2/4 \rfloor}.$
\end{thm}
We remark that for a graph $G$ with $e(G)=\lfloor n^2/4 \rfloor$, which is not $K_{\lfloor n/2 \rfloor, \lceil n/2  \rceil}$, $G$ contains at least one triangle. Therefore, the number of Gallai $r$-colorings of $G$ is at most $r(r + 2(r-1))r^{e(G)-3}<r^{\lfloor n^2/4 \rfloor}.$
As a direct consequence of Theorem~\ref{rgraph} and the above remark, we reprove Theorem~\ref{5extremal}, and in particular, we show that Conjecture~\ref{4extremal} is true for sufficiently large $n$.
\begin{thm}
There exists $n_0$ such that for all $n>n_0$, among all graphs of order $n$, the complete bipartite graph $K_{\lfloor n/2 \rfloor, \lceil n/2  \rceil}$ is the unique Gallai $4$-extremal graph. 
\end{thm}

\subsection{Overview of the paper}
Combining Szemer\'{e}di's Regularity Lemma and the stability method was used at many earlier works on extremal problems, including Erd\H{o}s-Rothchild type problems, see, e.g.,~\cite{ABKS, BBS, BHS, hoppen2017graphs}. 
However, our main approach relies on the method of hypergragh containers, developed independently by Balogh, Morris and Samotij~\cite{BMS} as well as by Saxton and Thomason~\cite{ST}, and some stability results for containers, which may be of independent interest to readers.

The paper is organized as follows.
First, in Section 2, we introduce some important definitions and then state a container theorem which is applicable to colorings.
In Section 3, we present a key enumeration result on the number of colorings with special restrictions, which will be used repeatedly in the rest of the paper.
Then in Section~4, we study the stability behavior of the containers for the complete graph, and apply the multicolor container theorem to give an asymptotic upper bound for the number of Gallai $r$-colorings of the complete graph. 
In Section 5, we deal with the Gallai $3$-colorings of dense non-complete graphs; the idea is the same as in Section 4 except that we need to provide a new stability result which is applicable to non-complete graphs. 

In the second half of the paper, that is, in Section 6, we study the Gallai $r$-colorings of non-complete graphs for $r\geq 4$. When the underlying graph is very dense, that is, close to the complete graph, we apply the same strategy as in Section 4 for the case $r=4$, where we prove a proper stability result for containers. The case $r\geq 5$ is even simpler, in which we actually prove that the number of Gallai colorings in each container is small enough. 
When the underlying graph has edge density close to the $\frac14$, i.e.~the edge density of the extremal graph, some new ideas are needed, and we also adopt a result of Bollob{\'a}s and Nikiforov~\cite{BN} on book graphs. 
For the rest of the graphs whose edge densities are between $\frac14 + o(1)$ and $\frac12 - o(1)$, we use a supersaturation result of triangle-free graphs given by Balogh, Bushaw, Collares, Liu, Morris, and Sharifzadeh~\cite{balogh2017typical}, and the above results on Gallai $r$-colorings for both high density graphs and low density graphs.\\

For a positive integer $n$, we write $[n]=\{1, 2, \ldots, n\}$.
For a graph $G$ and a set $A\subseteq V(G)$, the \textit{induced subgraph} $G[A]$ is the subgraph of $G$ whose vertex set is $A$ and whose edge set consists of all of the edges with both endpoints in $A$. 
For two disjoint subsets $A, B\subseteq V(G)$, the \textit{induced bipartite subgraph} $G[A, B]$ is the subgraph of $G$ whose vertex set is $A\cup B$ and whose edge set consists of all of the edges with one endpoint in $A$ and the other endpoint in $B$. 
Denote by $\delta(G)$ the minimum degree of $G$, and $\Delta(G)$ the maximum degree of $G$. 
For a graph $G$ and a vertex $v\in V(G)$, let $N_G(v)$ be the \textit{neighborhood} of $v$, i.e.~the set of vertices adjacent to $v$ in $G$, and $d_G(v)=|N_G(v)|$ be the \textit{degree} of $v$.
For a set $A\subseteq V(G)$, the \textit{neighborhood of $v$ restricted to $A$} is the set $N_G(v, A)=N_G(v)\cap A$; the \textit{degree of $v$ restricted to $A$}, denoted by $d_G(v, A)$, is the size of $N_G(v, A)$. When the underlying graph is clear, we simply write $N(v)$, $d(v)$, $N(v, A)$, and $d(v, A)$ instead. 
Throughout the paper, we omit all floor and ceiling signs whenever these are not crucial. 
Unless explicitly stated, all $n$-vertex graphs are assumed to be defined on the vertex set $[n]$, and all logarithms have base 2.  
\section{Preliminaries}
\subsection{The hypergraph container theorem}
We use the following version of the hypergraph container theorem (Theorem 3.1 in \cite{BS}). Let $\mathcal{H}$ be a $k$-uniform hypergraph with average degree $d$. The \textit{co-degree} of a set of vertices $S\subseteq V(\mathcal{H})$ is the number of edges containing $S$; that is,
\begin{equation*}
d(S) = \{e\in E(\mathcal{H}) \mid S \subseteq e\}.
\end{equation*}
For every integer $2\leq j\leq k$, the $j$-th maximum co-degree of $\mathcal{H}$ is 
\begin{equation*}
\Delta_j(\mathcal{H}) = \max\{d(S) \mid S \subseteq V(\mathcal{H}),\ |S| = j\}.
\end{equation*}
When the underlying hypergraph is clear, we simply write it as $\Delta_j$.
For $0< \tau < 1$, the \textit{co-degree function} $\Delta(\mathcal{H}, \tau)$ is defined as
\begin{equation*}
\Delta(\mathcal{H}, \tau) = 2^{\binom{k}{2}-1}\sum_{j=2}^{k}2^{-\binom{j-1}{2}}\frac{\Delta_j}{d \tau^{j-1}}.
\end{equation*}
In particular, when $k=3$, 
\begin{equation*}
\Delta(\mathcal{H}, \tau) = \frac{4\Delta_2}{d \tau} + \frac{2\Delta_3}{d \tau^2}.
\end{equation*}

\begin{thm}\label{HCT}\cite{BS}
Let $\mathcal{H}$ be a $k$-uniform hypergraph on vertex set $[N]$.  Let $0< \varepsilon, \tau <1/2$. Suppose that $\tau  <1/(200 k!^2k)$ and $\Delta(\mathcal{H}, \tau) \leq \varepsilon/(12k!)$.  Then there exists $c=c(k) \leq 1000k!^3k$ and a collection of vertex subsets $\mathcal{C}$ such that 
\begin{enumerate}[label={\upshape(\roman*)}]
\item every independent set in $\mathcal{H}$ is a subset of some $A\in  \mathcal{C}$;
\item for every $A\in \mathcal{C}$, $e(\mathcal{H}[A])\leq \varepsilon\cdot e(\mathcal{H})$;
\item $\log|\mathcal{C}| \leq cN\tau \log(1/\varepsilon) \log(1/\tau)$.
\end{enumerate} 
\end{thm}

\subsection{Definitions and multi-color container theorem}
A key tool in applying container theory to multi-colored structures will be the notion of a \textit{template}. This notion of `template', which was first introduced in~\cite{FOSU}, goes back to \cite{ST} under the name of `2-colored multigraphs' and later to \cite{BW}, where it is simply called `containers'.  For more studies about the multi-color container theory, we refer the interested reader to \cite{BMS, BMS18, BW, FOSU, ST}.

\begin{defi}[Template and palette]
An \textit{$r$-template} of order $n$ is a function $P : E(K_n) \rightarrow 2^{[r]}$, 
associating to each edge $e$ of $K_n$ a list of colors $P(e) \subseteq [r]$; we refer to this set $P(e)$ as the \textit{palette} available at $e$.
\end{defi}  

\begin{defi}[Subtemplate]
Let $P_1$, $P_2$ be two $r$-templates of order $n$. We say that $P_1$ is a \textit{subtemplate} of $P_2$ $($written as $P_1\subseteq P_2)$ if $P_1(e)\subseteq P_2(e)$ for every edge $e\in E(K_n)$.
\end{defi}

We observe that for $G\subseteq K_n$, an $r$-coloring of $G$ can be considered as an $r$-template of order $n$, with only one color allowed at each edge of $G$ and no color allowed at each non-edge. 
For an $r$-template $P$, write $\mathrm{RT}(P)$ for the number of subtemplates of $P$ that are rainbow triangles. We say that $P$ is \textit{rainbow triangle-free} if $\mathrm{RT}(P)=0$.
Using the container method, Theorem~\ref{HCT}, we obtain the following.

\begin{thm}\label{container}
For every $r\geq 3$, there exists a constant $c=c(r)$ and a collection $\mathcal{C}$ of $r$-templates of order $n$ such that
\begin{enumerate}[label={\upshape(\roman*)}]
\item every rainbow triangle-free $r$-template of order $n$ is a subtemplate of some $P\in  \mathcal{C}$;
\item for every $P\in \mathcal{C}$, $\mathrm{RT}(P)\leq n^{-1/3}\binom{n}{3}$;
\item $|\mathcal{C}| \leq 2^{cn^{-1/3}\log^2 n\binom{n}{2}}$.
\end{enumerate} 
\end{thm}
\begin{proof}
Let $\mathcal{H}$ be a 3-uniform hypergraph with vertex set $E(K_n)\times\{1, 2, \ldots, r\}$, whose edges are all triples $\{(e_1, d_1), (e_2, d_2), (e_3, d_3)\}$ such that $e_1, e_2, e_3$ form a triangle in $K_n$ and $d_1, d_2, d_3$ are all different. In other words, every hyperedge in $\mathcal{H}$ corresponds to a rainbow triangle of $K_n$. 
Note that there are exactly $r(r-1)(r-2)$ ways to rainbow color a triangle with $r$ colors. Hence, the average degree $d$ of $\mathcal{H}$ is equal to
$$d=\frac{3e(\mathcal{H})}{v(\mathcal{H})}=\frac{3r(r-1)(r-2)\binom{n}{3}}{r\binom{n}{2}}=(r-1)(r-2)(n-2).$$
For the application of Theorem~\ref{HCT}, let $\varepsilon = n^{-1/3}/r(r-1)(r-2)$ and $\tau = \sqrt{72\cdot3!\cdot r}n^{-1/3}$. 
Observe that $\Delta_2(\mathcal{H}) = r-2$, and $\Delta_3(\mathcal{H}) = 1$. For $n$ sufficiently large, we have
$\tau\leq 1/(200\cdot 3!^2\cdot 3)$ and 
$$\Delta(\mathcal{H}, \tau)= \frac{4(r-2)}{d \tau} + \frac{2}{d \tau^2}\leq \frac{3}{d \tau^2}\leq \frac{\varepsilon}{12\cdot 3!}.$$
Hence, there is a collection $\mathcal{C}$ of vertex subsets satisfying properties (i)-(iii) of Theorem~\ref{HCT}.
Observe that every vertex subset of $\mathcal{H}$ corresponds to an $r$-template of order $n$; every rainbow triangle-free $r$-template of order $n$ corresponds to an independent set in $\mathcal{H}$.
Therefore, $\mathcal{C}$ is a desired collection of $r$-templates.
\end{proof}

\begin{defi}[Gallai $r$-template]
For a graph $G$ of order $n$, an $r$-template $P$ of order $n$ is a \textit{Gallai $r$-template} of $G$ if it satisfies the following properties:
\begin{enumerate}[label={\upshape(\roman*)}]
\item for every $e\in E(G)$, $|P(e)|\geq 1$;
\item $\mathrm{RT}(P)\leq n^{-1/3}\binom{n}{3}.$
\end{enumerate}
\end{defi}

For a graph $G$ of order $n$ and a collection $\mathcal{P}$ of $r$-templates of order $n$, denote by $\mathrm{Ga}(\mathcal{P}, G)$ the set of Gallai $r$-colorings of $G$ which is a subtemplate of some $P\in \mathcal{P}$. 
If $\mathcal{P}$ consists of a single template $P$, then we simply write it as $\mathrm{Ga}(P, G)$.
\subsection{A technical lemma}
In this section, we provide a lemma that will be useful to us in what follows. We use a special case of the weak Kruskal-Katona theorem due to Lov\`{a}sz's~\cite{L}.
\begin{thm}[Lov\`{a}sz~\cite{L}]\label{trinum}
Suppose $G$ is a graph with $\binom{x}{2}$ edges, for some real number $x\geq 2$. Then the number of triangles of $G$ is at most $\binom{x}{3}$, with equality if and only if $x$ is an integer and $G=K_x$.
\end{thm}

\begin{lemma}\label{monoedge}
Let $n, r\in \mathbb{N}$ with $r\geq 3$ and $\frac{4}{n}-\frac{4}{n^2} \leq \varepsilon<\frac12$. If $G$ is an $r$-colored graph of order $n$, which contains at least $(1-\varepsilon)\binom{n}{3}$ monochromatic triangles, then there exists a color $c$ such that the number of edges colored by $c$ is at least $e(G)-4r^2\varepsilon\binom{n}{2}$.
\end{lemma}
\begin{proof}
We shall prove this lemma by contradiction. Let $\delta=4r^2\varepsilon$. Assume that none of the colors is used on at least $e(G)-\delta\binom{n}{2}$ edges. 

First, we conclude that $e(G)\geq (1-\varepsilon)\binom{n}{2}$. If not, then by Theorem~\ref{trinum}, the number of triangles of $G$ is less than 
$$\frac{\sqrt{2}}{3}(1-\varepsilon)^{3/2}\binom{n}{2}^{3/2}\leq (1-\varepsilon)\binom{n}{3},$$
which contradicts the assumption.

By the pigeonhole principle, we can assume without loss of generality that the set of red edges in $G$, denoted by $\mathrm{R}(G)$, satisfies
$|\mathrm{R}(G)|\geq(1-\varepsilon)\binom{n}{2}/r.$
By the contradiction assumption, we have $|\mathrm{R}(G)|<e(G)-\delta\binom{n}{2}$. Therefore, the number of non-red edges is greater than $\delta\binom{n}{2}$. Again, without loss of generality, we can assume that the set of blue edges in $G$, denoted by $\mathrm{B}(G)$, satisfies
$|\mathrm{B}(G)|\geq\delta\binom{n}{2}/r.$

For an edge in $\mathrm{R}(G)$ and an edge in $\mathrm{B}(G)$, these two edges either share one endpoint or are vertex disjoint, see Figure~\ref{figA}. In the first case, see Figure~\ref{fig:A1}, the triple $abc$ could not form a monochromatic triangle of $G$. In the latter case, see Figure~\ref{fig:A2}, at least one of $abc$ and $bcd$ is not a monochromatic triangle of $G$.
\begin{figure}[H]
\centering       
\begin{subfigure}{0.32\textwidth} 
\centering          
\resizebox{\linewidth}{!}{ 
\begin{tikzpicture}[thick, acteur/.style={circle, fill=black, thick, inner sep=2pt, minimum size=0.2cm}]
     
\node (1) at (0,2) [acteur,label=a]{};
\node (2) at (-2,0) [acteur,label=below:b]{};
\node (3) at (2, 0) [acteur,label=below:c]{};

\node at (-1.4, 0.8) [label=red]{};
\node at (1.4, 0.8) [label=blue]{};

\draw [line width=0.3mm, black] (1) to (2);
\draw [line width=0.3mm, black](1) to (3);

\draw [line width=0.3mm, black, dashed] (2) to (3);
\end{tikzpicture}
}
\caption{}
\label{fig:A1}
\end{subfigure}
\hspace{1.5cm}
\begin{subfigure}{0.4\textwidth} 
\centering   
\resizebox{\linewidth}{!}{ 
\begin{tikzpicture}[thick, acteur/.style={circle, fill=black, thick, inner sep=2pt, minimum size=0.2cm}]
     
\node (1) at (-2,2) [acteur,label=a]{};
\node (2) at (-2,0) [acteur,label=below:b]{};
\node (3) at (2, 2) [acteur,label=c]{};
\node (4) at (2, 0) [acteur,label=below:d]{};

\node at (-2.5, 0.6) [label=red]{};
\node at (2.6, 0.6) [label=blue]{};

\draw [line width=0.3mm, black] (1) to (2);
\draw [line width=0.3mm, black] (3) to (4);

\draw [line width=0.3mm, black, dashed] (2) to (3);
\end{tikzpicture}
}
\caption{}
\label{fig:A2}
\end{subfigure}
\caption{Two cases of a red-blue pair of edges.}
\label{figA}
\end{figure}
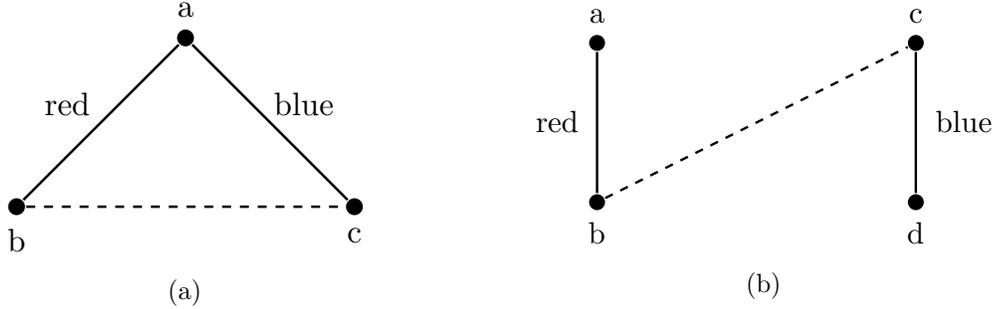

Let $\mathrm{NT}(G)$ be the family of triples $\{a, b, c\}$ which does not form a monochromatic triangle of $G$. The above discussion shows that each pair of red and blue edges generates at least one triple in $\mathrm{NT}(G)$.  Observe that each triple in $\mathrm{NT}(G)$ can be counted in at most $2 + 3(n-3)$ pairs of red and blue edges. Hence, we obtain that
$$|\mathrm{NT}(G)|\geq \frac{(1-\varepsilon)\binom{n}{2}/r\cdot \delta\binom{n}{2}/r}{2 + 3(n-3)}> \frac{\delta}{4r^2}\binom{n}{3}= \varepsilon\binom{n}{3},$$
which contradicts the assumption of the lemma.
\end{proof}

\section{Counting Gallai colorings in $r$-templates}
In this section, we aim to prove the following technical theorem, which will be used repeatedly in the rest of the paper.
\begin{thm}\label{2colorclass}
Let $n, r\in \mathbb{N}$ with $r\geq 3$, and $G$ be a graph of order $n$.
Suppose that $\delta = \log^{-11} n$ and $k$ is a positive constant, which does not depend on $n$. For two colors $i$, $j\in [r]$, denote by $\mathcal{F}=\mathcal{F}(i, j)$ the set of $r$-templates of order $n$, which contain at least $(1- k\delta)\binom{n}{2}$ edges with palette $\{i, j\}$. Then, for $n$ sufficiently large, 
$$|\mathrm{Ga}(\mathcal{F}, G)|\leq 2^{e(G)} + 2^{-\frac{n}{3\log^ 2n}}2^{\binom{n}{2}}.$$
\end{thm}

Fix two colors $1\leq i< j \leq r$, and let $S= [r]-\{i, j\}$.
For an $r$-coloring $F$ of $G$, let $S(F)$ be the set of edges in $G$, which are colored by colors in $S$.
From the definition of $\mathcal{F}$, we immediately obtain the following proposition.
\begin{prop}\label{bedge}
For every $F\in \mathrm{Ga}(\mathcal{F}, G)$, the number of edges in $S(F)$ is at most $k\delta\binom{n}{2}$.
\end{prop}

\begin{lemma}\label{lmatching}
Let $\mathcal{F}_1$ be the set of $F\in \mathrm{Ga}(\mathcal{F}, G)$ such that $S(F)$ contains a matching of size $\delta n\log^2 n$. Then, for $n$ sufficiently large,
 $$|\mathcal{F}_1|\leq 2^{-\frac{n^2}{5\log^9 n}}2^{\binom{n}{2}}.$$
\end{lemma}

\begin{proof}
Let us consider the ways to color $G$ so that the resulting colorings are in $\mathcal{F}_1$. 
We first choose the set of edges $E^S$ which will be colored by the colors in $S$. Note that $E^S$ must contain a matching of size $\delta n\log^2 n$ by the definition of $\mathcal{F}_1$.
By Proposition~\ref{bedge}, there are at most $\sum_{i\leq k\delta\binom{n}{2}}\binom{\binom{n}{2}}{i}$ choices for such $E^S$, and the number of ways to color them is at most $r^{k\delta\binom{n}{2}}.$ 
In the next step, take a matching $M$ of size $\delta n\log^2 n$ in $E^S$; the number of ways to choose such matching is at most $\binom{\binom{n}{2}}{\delta n\log^2 n}$. 

Let $A=V(M)$ and $B=[n]\setminus A$. Denote by $\mathcal{T}$ the set of triangles of $K_n$ with a vertex in $B$ and an edge from $M$, which contain no edge in $E^S\cap G[A, B]$. 
We claim that $|\mathcal{T}|\geq \frac14\delta n^2\log^2 n$ as otherwise we would obtain that 
\[|E^S| \geq |B|\cdot \delta n\log^2 n - |\mathcal{T}| + |M|\geq \frac12\delta n^2\log^2 n - \frac14\delta n^2\log^2 n=\frac14\delta n^2\log^2 n > k\delta\binom{n}{2},\]
which, by Proposition \ref{bedge}, contradicts the fact that $F\in \mathrm{Ga}(\mathcal{F}, G)$.
Note that if a triangle $T$ in $\mathcal{T}$ contains more than one uncolored edge, then they must have the same color in order to avoid the rainbow triangle. Hence, the number of ways to color the uncolored edges in $\mathcal{T}$ is at most $2^{|\mathcal{T}|}$. 

There remain at most $\binom{n}{2} - 2|\mathcal{T}|$ uncolored edges and they can only be colored by $i$ or $j$, as edges in $E_S$ are already colored. Hence, the number of ways to color the rest of edges is at most
$2^{\binom{n}{2} - 2|\mathcal{T}|}.$
In conclusion, we obtain that 
 \begin{equation*}
  \begin{aligned}
 |\mathcal{F}_1| &  \leq \textstyle\sum_{i\leq k\delta\binom{n}{2}}\binom{\binom{n}{2}}{i}r^{k\delta\binom{n}{2}}
\binom{\binom{n}{2}}{\delta n\log^2 n}
 \cdot 2^{|\mathcal{T}|}
\cdot 2^{\binom{n}{2} - 2|\mathcal{T}|} \\
& \leq 2^{O(\delta n^2 \log n)}
\cdot 2^{O(\delta n \log^3 n)}
 \cdot 2^{\binom{n}{2} - \frac14\delta n^2\log^2 n}
\leq 2^{\binom{n}{2}-\frac{n^2}{5\log^9 n}}.
\end{aligned}
\end{equation*}
\end{proof}

\begin{lemma}\label{tmatching}
For every integer $1\leq t < \delta n\log^2 n$, let $\mathcal{F}(t)$ be the set of $F\in \mathrm{Ga}(\mathcal{F}, G)$, in which the maximum matching of $S(F)$ is of size $t$. Then, for $n$ sufficiently large, $$|\mathcal{F}(t)|\leq 2^{-\frac{n}{2\log^2n}}2^{\binom{n}{2}}.$$
\end{lemma}

\begin{proof}
For a fixed $t$, let us count the ways to color $G$ so that the resulting colorings are in $\mathcal{F}(t)$. 
By the definition of $\mathcal{F}(t)$, among all edges which will be colored by the colors in $S$, there exists a maximum matching $M$ of size $t$. 
We first choose such matching; the number of ways is at most $\binom{\binom{n}{2}}{t}$. 
Once we fix the matching $M$, let $A = V(M)$ and $B= [n]\setminus A$. By the maximality of $M$, we immediately obtain the following claim.
\begin{claim}\label{claim2color}
None of the edges in $G[B]$ can be colored by the colors in $S$. \qed
\end{claim}
Denote by $\mathrm{Cr}(S)$ the set of edges in $G[A, B]$ which will be colored by the colors in $S$. For a vertex $u\in A$, denote by $\mathrm{Cr}(S, u)$ the set of edges in $\mathrm{Cr}(S)$ with one endpoint $u$. Similarly, define $\mathrm{Cr}(\{i, j\}, u)$ to be the set of edges in $G[u, B]$ which will be colored by the colors $i$ or $j$.
We shall divide the proof into three cases.\\

\noindent\textbf{Case 1}: $|\mathrm{Cr}(S)|\leq \frac{nt}{\log^2 n}$. \\
We first color the edges in $G[A]$ and the number of options is at most $r^{\binom{2t}{2}}$.
In the next step, we select and color the edges in $\mathrm{Cr}(S)$; by the above inequality, the number of ways is at most $\sum_{i\leq \frac{nt}{\log^2 n}}\binom{2nt}{i}r^{\frac{nt}{\log^2 n}}$. 
By Claim~\ref{claim2color}, the remaining edges can only use the colors $i$ or $j$.
Let $\mathcal{T}$ be the set of triangles of $K_n$ formed by a vertex in $B$ and an edge from $M$, which contain no edge in $\mathrm{Cr}(S)$. We claim that $|\mathcal{T}|\geq \frac14nt$ as otherwise we would obtain 
\[|\mathrm{Cr}(S)|\geq |B|t - |\mathcal{T}| \geq\frac12 nt -\frac14 nt >\frac{nt}{\log^2 n},\]
which contradicts the assumption.
If a triangle $T$ in $\mathcal{T}$ contains more than one uncolored edge, then they must have the same color in order to avoid the rainbow triangle. Hence, the number of ways to color the uncolored edges in $\mathcal{T}$ is at most $2^{|\mathcal{T}|}$. 

There remain at most $\binom{n}{2} - 2|\mathcal{T}|-\binom{2t}{2}$ uncolored edges, and they can be colored by $i$ or $j$. Therefore the number of ways to color the rest of the edges is at most
$2^{\binom{n}{2}-2|\mathcal{T}|-\binom{2t}{2}}.$
In conclusion, we obtain that the number of $r$-coloring $F\in \mathcal{F}(t)$ with $|\mathrm{Cr}(S)|\leq \frac{nt}{\log^2 n}$ is at most
 \begin{equation*}
  \begin{split}
&\binom{\binom{n}{2}}{t}
\cdot r^{\binom{2t}{2}}
\cdot \textstyle\sum_{i\leq \frac{nt}{\log^2 n}}\binom{2nt}{i}r^{\frac{nt}{\log^2 n}}
\cdot 2^{|\mathcal{T}|}
\cdot 2^{\binom{n}{2}-2|\mathcal{T}|-\binom{2t}{2}}\\
\leq & 2^{O(t\log n)}
\cdot 2^{O(t^2)}
\cdot 2^{O\left(\frac{nt}{\log n}\right)}
\cdot 2^{\binom{n}{2}-\frac14nt}
\leq 2^{\binom{n}{2}-\frac15nt}
\leq 2^{\binom{n}{2}-\frac15n},
\end{split}
\end{equation*}
where the third inequality is given by $t^2\leq t\cdot \delta n\log^2 n=nt/\log^9 n.$
\\

\noindent\textbf{Case 2}: There exists a vertex $u\in A$ such that 
\begin{equation}\label{midde}
|\mathrm{Cr}(S, u)|\geq \frac{n}{\log^4 n} 
\quad\quad\text{and}\quad\quad
|\mathrm{Cr}(\{i, j\}, u)|\geq \frac{n}{\log^4 n}.
\end{equation}
We first choose the vertex $u$, and the number of options is at most $2t$. Moreover, the number of ways to select and color edges in $\mathrm{Cr}(S, u)$ is at most $r^n2^n$.
In the next step, we color all the uncolored edges in $G[A, B]$ and $G[A]$, and the number of ways is at most $r^{2nt+\binom{2t}{2}}$. 
Let $\mathcal{T}$ be the set of triangles $T=\{uvw\}$ of $K_n$, in which $v, w\in B$, $uv\in \mathrm{Cr}(S, u)$, and $uw\in \mathrm{Cr}(\{i, j\}, u)$. By the relation (\ref{midde}), we have $|\mathcal{T}|\geq \frac{n^2}{\log^8 n}$.
For every triangle $T=\{uvw\}\in\mathcal{T}$, if $vw$ is an edge of $G$, then by Claim~\ref{claim2color} it can only be colored by $i$ or $j$, and 
must have the same color with $uw$ in order to avoid the rainbow triangle.
Therefore, the number of ways to color the uncolored edges in $\mathcal{T}$ is $1$. 

There remain at most $\binom{n}{2} - |\mathcal{T}|$ uncolored edges in $B$, as other edges are already colored. By Claim~\ref{claim2color}, none of the remaining edges in $B$ could use the colors from $S$. Therefore, the number of ways to color the rest of edges is at most 
$2^{\binom{n}{2}-|\mathcal{T}|}.$
In conclusion, we
obtain that the number of $F\in \mathcal{F}(t)$ which is included in Case 2 is at most
 \begin{equation*}
\binom{\binom{n}{2}}{t}
\cdot 2t \cdot r^n2^n
\cdot r^{2nt+\binom{2t}{2}}
\cdot 2^{\binom{n}{2}-|\mathcal{T}|}
\leq 2^{O(t\log n)}
\cdot 2^{O(n)}
\cdot 2^{O(nt)}
\cdot 2^{\binom{n}{2}-\frac{n^2}{\log^8 n}}
\leq 2^{\binom{n}{2}-\frac{n^2}{2\log^8 n}},
\end{equation*}
where the last inequality is given by the condition that $nt\leq n\cdot \delta n\log^2n = n^2/\log^9n.$\\

\noindent\textbf{Case 3:} $|\mathrm{Cr}(S)|> \frac{nt}{\log^2 n}$, and for every vertex $u\in A$, 
\begin{equation}\label{exde}
|\mathrm{Cr}(S, u)|< \frac{n}{\log^4 n} 
\quad\quad\text{or}\quad\quad
|\mathrm{Cr}(\{i, j\}, u)|< \frac{n}{\log^4 n}.
\end{equation}
We first color the edges in $G[A]$ and the number of ways is at most $r^{\binom{2t}{2}}$. By~(\ref{exde}), for every vertex $u\in A$, the number of ways to select $\mathrm{Cr}(S, u)$ is at most $2\sum_{i\leq n/\log^4 n}\binom{n}{i}\leq 2^{n/\log^3 n}$. Therefore, the number of ways to select $\mathrm{Cr}(S)$ is at most $2^{2nt/\log^3 n}$.\\

\noindent\textbf{Subcase 3.1: }$e(G)\leq \binom{n}{2} - \frac{n^2}{4\log^6 n}.$\\
The number of ways to color $\mathrm{Cr}(S)$ is at most $r^{2nt}$.  By Claim~\ref{claim2color}, the rest of the edges can only be colored by $i$ or $j$, and the number of them is at most $e(G) - |\mathrm{Cr}(S)|$. Hence, the number of $F\in \mathcal{F}(t)$ covered in Case 3.1 is at most
 \begin{equation*}
\binom{\binom{n}{2}}{t}
\cdot r^{\binom{2t}{2}}
\cdot 2^{\frac{2nt}{\log^3 n}}
\cdot r^{2nt}
\cdot 2^{e(G)-|\mathrm{Cr}(S)|}
\leq  2^{O(t\log n)}
\cdot 2^{O(nt)}
\cdot 2^{\binom{n}{2} - \frac{n^2}{4\log^6 n}-\frac{nt}{\log^2 n}}
\leq 2^{\binom{n}{2}- \frac{n^2}{5\log^6 n}},
\end{equation*}
where the last inequality holds by the condition that $nt\leq n\cdot \delta n\log^2n = n^2/\log^9n.$\\

\noindent\textbf{Subcase 3.2: }$e(G)> \binom{n}{2} - \frac{n^2}{4\log^6 n}.$\\
For $u\in A$, define $N_S(u)=\{v\in B \mid uv\in \mathrm{Cr}(S, u)\}$. Let $G_u$ be the induced subgraph of $G$ on $N_S(u)$, and denote by $c(G_u)$ the number of components of $G_u$.
\begin{claim}\label{claim: num_comp}
For every $u\in A$, we have $c(G_u)\leq \frac{n}{\log^3 n}$.
\end{claim}
\noindent\textit{Proof.}
Suppose that there exists a vertex $u$ in $A$ with $c(G_u)>\frac{n}{\log^3 n}$. Then the number of non-edges in $G_u$ is at least $\binom{\frac{n}{\log^3 n}}{2}\geq \frac{n^2}{4\log^6 n}$, which contradicts with the assumption of Case 3.2.
\qed
\begin{claim}\label{claim: num_color}
For every $u\in A$, the number of ways to color $\mathrm{Cr}(S, u)$ is at most $r^{c(G_u)}$.
\end{claim}
\noindent\textit{Proof.} Let $C$ be an arbitrary component of $G_u$. It is sufficient to prove that for every $v, w\in V(C)$, $uv$ and $uw$ must have the same color. Assume that there exist $v, w\in V(C)$ such that $uv$ and $uw$ receive different colors. Since $C$ is a connected component of $G_u$, there is a path $P=\{ v=v_0, v_1, v_2, \ldots, v_k=w\}$ in $G_u$,  in which $uv_i$ is painted by a color in $S$ for every $0\leq i\leq k$. Moreover, since $uv$ and $uw$ receive different colors, there exists an integer $0\leq j\leq k-1$ such that $uv_j$ and $uv_{j+1}$ receive different colors. On the other hand, by Claim~\ref{claim2color}, $v_jv_{j+1}$ can only be colored by $i$ or $j$. Therefore, $u, v_j, v_{j+1}$ form a rainbow triangle, which is not allowed in a Gallai $r$-coloring.\qed 

By Claims~\ref{claim: num_comp} and \ref{claim: num_color}, the number of ways to color $\mathrm{Cr}(S, u)$ is at most $r^{\frac{n}{\log^3 n}}$, and therefore the total number of ways to color $\mathrm{Cr}(S)$ is at most $r^{\frac{2nt}{\log^3 n}}$. By Claim~\ref{claim2color}, the rest of the edges can only be colored by $i$ or $j$, and the number of them is at most $e(G) - |\mathrm{Cr}(S)|$. Hence, the number of $F\in \mathcal{F}(t)$ included in Case 3.2 is at most
 \begin{equation*}
\binom{\binom{n}{2}}{t}
\cdot r^{\binom{2t}{2}}
\cdot 2^{\frac{2nt}{\log^3 n}}
\cdot r^{\frac{2nt}{\log^3 n}}
\cdot 2^{e(G)-|\mathrm{Cr}(S)|}
\leq  2^{O(t\log n)}
\cdot 2^{O\left(\frac{nt}{\log^3 n}\right)}
\cdot 2^{\binom{n}{2} -\frac{nt}{\log^2 n}}
\leq 2^{\binom{n}{2}-\frac{n}{2\log^2 n} - 1}.
\end{equation*}

Eventually, we conclude that 
\[
|\mathcal{F}(t)|\leq 2^{\binom{n}{2}-\frac15n} + 2^{\binom{n}{2}-\frac{n^2}{2\log^8 n}} + 2^{\binom{n}{2}-\frac{n}{2\log^2n} -1}\leq 2^{-\frac{n}{2\log^2 n}}2^{\binom{n}{2}}
\] for every $1\leq t < \delta n\log^2 n$.
\end{proof}

Observe that every $r$-coloring of $G$ using at most 2 colors is a Gallai $r$-coloring. Then we immediately obtain the following lemma.
\begin{lemma}\label{0matching}
Let $\mathcal{F}_0$ be the set of $F\in \mathrm{Ga}(\mathcal{F}, G)$ such that $S(F)=\emptyset$. Then $|\mathcal{F}_0|= 2^{e(G)}$.
\end{lemma}
Now, we have all the ingredients to prove Theorem~\ref{2colorclass}.
\\

\noindent\textit{Proof of Theorem~\ref{2colorclass}.}
Applying Lemmas~\ref{lmatching}, ~\ref{tmatching} and~\ref{0matching}, we obtain that
$$|\mathrm{Ga}(\mathcal{F}, G)|= |\mathcal{F}_1| + \sum_{t=1}^{\delta n/\log^2 n}|\mathcal{F}(t)|+|\mathcal{F}_0| \leq 2^{e(G)}+2^{-\frac{n}{3\log^2 n}}2^{\binom{n}{2}},$$
for $n$ sufficiently large.
\qed

\section{Gallai $r$-colorings of complete graphs}

%

\subsection{Stability of the Gallai $r$-template of complete graphs}\label{sec:complete:property}

\begin{prop}\label{less3color}
Let $n$, $r\in \mathbb{N}$ with $r\geq 3$. Suppose $P$ is a Gallai $r$-template of $K_n$. Then the number of edges with at least 3 colors in its palette is at most $n^{-1/6}n^2$.
\end{prop}

\begin{proof}
Let $E=\{e\in E(K_n):\ |P(e)|\geq 3\}$ and assume that $|E|>n^{-1/6}n^2$.
Let $F$ be a spanning subgraph of $K_n$ with edge set $E$. For every $i\in [n]$, denote by $d_i$ the degree of vertex $i$ of $F$. Then the number of 3-paths in $F$ is equal to 
$$\sum_{i\in [n]}\binom{d_i}{2}
\geq n \binom{\frac{\sum_{i\in [n]}d_i}{n}}{2}
\geq n\binom{2|E|/n}{2}\geq \frac{|E|^2}{n}>3n^{-1/3}\binom{n}{3}.$$
Observe that if $i, j, k$ is a 3-path in $F$, then there is at least one rainbow triangle in $P$ with vertex set $\{i, j, k\}$ since edges $ij$, $jk$ have at least 3 colors in its palette and edge $ik$ has at least one color in its palette. Therefore, there would be more than $n^{-1/3}\binom{n}{3}$ rainbow triangles in $P$, which contradicts the fact that $P$ is a Gallai $r$-template.
\end{proof}

\begin{lemma}\label{baltri}
 Let $n$, $r\in \mathbb{N}$ with $r\geq 3$ and $n^{-1/6} \ll \delta \ll 1$. Assume that $P$ is a Gallai $r$-template of $K_n$ with $|\mathrm{Ga}(P, K_n)|>2^{(1- \delta)\binom{n}{2}}$. Then the number of triangles $T$ of $K_n$ with $\sum_{e\in T}|P(e)|=6$ and $P(e)=P(e')$ for every $e$, $e'\in T$ is at least $(1 - 4\delta)\binom{n}{3}$.
\end{lemma}
\begin{proof}
Let $\mathcal{T}$ be the collection of triangles of $K_n$. We define
\begin{equation*}
\arraycolsep=1.4pt\def\arraystretch{1.2}
\begin{array}{ll}
\mathcal{T}_1&=\left\{T \in \mathcal{T} \mid \textstyle\sum_{e\in T}|P(e)|=6 \text{ and } P(e)=P(e')\text{ for every }e,~e'\in T\right\},\\
\mathcal{T}_2&=\left\{T \in \mathcal{T}\mid \exists~e\in T, \ |P(e)|\geq 3\right\},\\
\mathcal{T}_3&=\left\{T \in \mathcal{T}\setminus (\mathcal{T}_1 \cup \mathcal{T}_2) \mid \textstyle\sum_{e\in T}|P(e)| = 6\right\},\\
\mathcal{T}_4&=\left\{T \in \mathcal{T}\setminus \mathcal{T}_2 \mid \textstyle\sum_{e\in T}|P(e)|\leq 5\right\}.
\end{array}
\end{equation*}
Let $|\mathcal{T}_1|=\alpha\binom{n}{3}$, $|\mathcal{T}_2|=\beta\binom{n}{3}$, $|\mathcal{T}_3|=\gamma\binom{n}{3}$. Then $|\mathcal{T}_4| \leq(1 - \alpha)\binom{n}{3}$.
By Proposition~\ref{less3color}, we have $|\mathcal{T}_2|\leq n^{-1/6}n^3$ and therefore $\beta\leq 12n^{-1/6}$. Observe that for every $T\in \mathcal{T}_3$, the template $P$ contains a rainbow triangle with edge set $T$; therefore, we obtain that $|\mathcal{T}_3|\leq \mathrm{RT}(P)\leq n^{-1/3}\binom{n}{3}$, which gives $\gamma \leq n^{-1/3}\leq n^{-1/6}$.

Assume that $\alpha< 1-4\delta$. Then the number of Gallai $r$-colorings of $K_n$, which are subtemplates of $P$,  satisfies 
\begin{equation*}
\arraycolsep=1.4pt\def\arraystretch{1.2}
\begin{array}{ll}
\log |\mathrm{Ga}(P, K_n)| & \leq \log \left(\prod_{e\in E(K_n)}|P(e)|\right) = \log \left(\prod_{T\in \mathcal{T}}\prod_{e\in T}|P(e)|  \right)^{\frac{1}{n-2}} \\
& \leq \log\left(\prod_{T\in \mathcal{T}_1}2^3 \prod_{T\in \mathcal{T}_2}r^3 \prod_{T\in \mathcal{T}_3}2^3 \prod_{T\in \mathcal{T}_4}2^2\right)\cdot \frac{1}{n-2}\\
&\leq \left(3\alpha + 3\beta\log r  + 3\gamma + 2(1-\alpha)\right)\frac13\binom{n}{2}\\
&\leq \left(2 +\alpha + (36\log r +3)n^{-1/6}\right)\frac13\binom{n}{2}
< \left(2 + (1-4\delta) + \delta\right)\frac13\binom{n}{2}
= (1- \delta)\binom{n}{2}.
\end{array}
\end{equation*}
This contradicts the assumption that $|\mathrm{Ga}(P, K_n)|>2^{(1- \delta)\binom{n}{2}}$.
\end{proof}

We now prove a stability result for Gallai $r$-templates of $K_n$.
\begin{thm}\label{2edge}
 Let $n$, $r\in\mathbb{N}$ with $r\geq 3$ and $n^{-1/6} \ll \delta \ll 1$. Assume that $P$ is a Gallai $r$-template of $K_n$ with $|\mathrm{Ga}(P, K_n)|>2^{(1- \delta)\binom{n}{2}}$. Then there exist two colors $i$, $j\in [r]$ such that the number of edges of $K_n$ with palette $\{i, j\}$ is at least $(1- 4r^4\delta)\binom{n}{2}$.
\end{thm}
\begin{proof}
Let $G$ be an $\binom{r}{2}$-colored graph with edge set $E(G) = \{e\in E(K_n) \mid |P(e)|=2\}$ and color set $\{(i, j) \mid 1\leq i< j\leq r\}$, where each edge $e$ is colored by color $P(e)$. 
By Lemma~\ref{baltri}, the number of monochromatic triangles in $G$ is at least $(1-4\delta)\binom{n}{3}$. Applying Lemma \ref{monoedge} on $G$, we obtain that there exist two colors $i$, $j$ such that  the number of edges with palette $\{i, j\}$ is at least $e(G) - 4\binom{r}{2}^2 \cdot 4\delta\binom{n}{2}\geq (1-4\delta)\binom{n}{2} - 4\binom{r}{2}^2 \cdot 4\delta\binom{n}{2} \geq (1- 4r^4\delta)\binom{n}{2}$. 
\end{proof}

\subsection{Proof of Theorem~\ref{rcomplete}}
\noindent\textit{Proof of Theorem~\ref{rcomplete}.}
Let $\mathcal{C}$ be the collection of containers given by Theorem \ref{container}. We observe that a Gallai $r$-coloring of $K_n$ can be regarded as a rainbow triangle-free $r$-coloring template of order $n$, with only one color allowed at each edge. Therefore, by Property (i) of Theorem \ref{container}, every Gallai $r$-coloring of $K_n$ is a subtemplate of some $P\in \mathcal{C}$.

Let $\delta = \log^{-11} n$. We define
\[
\mathcal{C}_1 = \left\{P\in \mathcal{C} :\ |\mathrm{Ga}(P, K_n)|\leq 2^{(1- \delta)\binom{n}{2}}\right\},
\hspace{3mm}
\mathcal{C}_2 = \left\{P\in \mathcal{C} :\ |\mathrm{Ga}(P, K_n)|> 2^{(1- \delta)\binom{n}{2}}\right\}.
\]
By Property (iii) of Theorem~\ref{container}, we have 
$$|\mathrm{Ga}(\mathcal{C}_1, K_n)|\leq |\mathcal{C}_1|\cdot 2^{(1-\delta)\binom{n}{2}}\leq 2^{c n^{-1/3}\log^2 n\binom{n}{2}}\cdot2^{\binom{n}{2} - \log^{-11}n \binom{n}{2}}\leq 2^{-\frac{n^2}{4\log^{11} n}}2^{\binom{n}{2}}.$$

We claim that every template $P$ in $\mathcal{C}_2$ is a Gallai $r$-template of $K_n$. 
First, by Property (ii) of Theorem~\ref{container}, we have $\mathrm{RT}(P)\leq n^{-1/3}\binom{n}{3}$. Suppose that there exists an edge $e\in E(K_n)$ with $|P(e)|=0$. Then we would obtain $\mathrm{Ga}(P, K_n) = \emptyset$ as a Gallai $r$-coloring of $K_n$ requires at least one color on each edge, which contradicts the definition of $\mathcal{C}_2$. 
Now by Theorem~\ref{2edge}, we can divide $\mathcal{C}_2$ into classes $\{\mathcal{F}_{i, j}, 1\leq i<j\leq r\}$, where $\mathcal{F}_{i, j}$ consists of all the $r$-templates in $\mathcal{C}_2$ which contain at least $(1-4r^4\delta)\binom{n}{2}$ edges with palette $\{i, j\}$. Applying Theorem~\ref{2colorclass} on $\mathcal{F}_{i, j}$, we obtain that $|\mathrm{Ga}(\mathcal{F}_{i, j}, K_n)|\leq \left(1+2^{-\frac{n}{3\log^2 n}}\right)2^{\binom{n}{2}}$, and therefore
\begin{equation*}
|\mathrm{Ga}(\mathcal{C}_2, K_n)|\leq \sum_{1\leq i<j\leq r}|\mathrm{Ga}(\mathcal{F}_{i, j}, K_n)|\leq\binom{r}{2}\left(1+2^{-\frac{n}{3\log^2 n}}\right)2^{\binom{n}{2}}.
\end{equation*}
Finally, we conclude that 
$$|\mathrm{Ga}(\mathcal{C}, K_n)|= |\mathrm{Ga}(\mathcal{C}_1, K_n)| + |\mathrm{Ga}(\mathcal{C}_2, K_n)|\leq \left(\binom{r}{2}+2^{-\frac{n}{4\log^2 n}}\right)2^{\binom{n}{2}},$$ which gives the desired upper bound for the number of Gallai $r$-colorings of $K_n$. 
\qed
\section{Gallai $3$-colorings of non-complete graphs}
In this section, we count Gallai 3-colorings of dense non-complete graphs. We shall explore the stability property first, and then follow a somewhat similar strategy as in the proof of Theorem~\ref{rcomplete}.
The main obstacle is that in a Gallai $r$-template of a non-complete graph, a palette of an edge could be an empty set, which leads to a more sophisticated discussion of templates.

\subsection{Triangles in $r$-templates of dense graphs}
Let $\mathcal{T}$ be the collection of triangles of $K_n$. For a given $r$-template $P$ of order $n$, we partition the triangles into 5 classes. We set an extra class, as a $T\in \mathcal{T}$ may not be a triangle in $G$.
\begin{equation}\label{eq: tripartition}
\arraycolsep=1.4pt\def\arraystretch{1.2}
\begin{array}{ll}
\mathcal{T}_1(P)&=\left\{T \in \mathcal{T} \mid \textstyle\sum_{e\in T}|P(e)|=6 \text{ and } P(e)=P(e')\text{ for every }e,~e'\in T\right\},\\
\mathcal{T}_2(P)&=\left\{T \in \mathcal{T} \mid T=\{e_1, e_2, e_3\},~|P(e_1)|\geq 3,~|P(e_2)|\geq 3,~\text{and}~|P(e_3)|=0\right\},\\
\mathcal{T}_3(P)&=\left\{T \in \mathcal{T} \mid T=\{e_1, e_2, e_3\},~|P(e_1)|\geq 3,~|P(e_2)|+|P(e_3)|\leq 2\right\},\\
\mathcal{T}_4(P)&=\left\{T \in \mathcal{T}\setminus (\mathcal{T}_1\cup \mathcal{T}_2\cup \mathcal{T}_3) \mid  \textstyle\sum_{e\in T}|P(e)| \geq 6\right\},\\
\mathcal{T}_5(P)&=\left\{T \in \mathcal{T}\setminus \mathcal{T}_3 \mid  \textstyle\sum_{e\in T}|P(e)|\leq 5\right\}.
\end{array}
\end{equation}

\begin{lemma}\label{small330}
Let $n, r\in \mathbb{N}$ with $r\geq 4$ and $0< k\leq 1$. For $0<\xi \leq \left(\frac{k}{2 + 6k}\right)^2$, let $G$ be a graph of order $n$, and $e(G)\geq (1-\xi)\binom{n}{2}$. Assume that $P$ is a Gallai $r$-template of $G$. Then, for sufficiently large $n$,
$$|\mathcal{T}_2(P)|\leq \max\left\{k |\mathcal{T}_3(P)|,~\frac{3 + 9k}{k}n^{-\frac13}\binom{n}{3}\right\}.$$
\end{lemma}

\begin{proof}
Let $E=\{e\in E(K_n):\ |P(e)|\geq3\}$ and $F$ be a spanning subgraph of $K_n$ with edge set $E$. For every $i\in [n]$, denote by $d_i$ the degree of vertex $i$ of $F$. 
Since $\sum_{i=1}^n d_i=2|E|$, the number of vertices with $d_i>\sqrt{\xi}n$ is less than $\frac{2|E|}{\sqrt{\xi}n}$. Therefore, we obtain 
\begin{equation}\label{eq: T2class}
\begin{split}
|\mathcal{T}_2(P)| & \leq \sum_{i=1}^{n}\min\left\{\binom{d_i}{2}, \xi\binom{n}{2}\right\}< \frac{2|E|}{\sqrt{\xi}n}\cdot \frac{\xi n^2}{2} + \sum_{d_i\leq \sqrt{\xi}n}\frac{d_i^2}{2}\\
&\leq \frac{2|E|}{\sqrt{\xi}n}\cdot \frac{\xi n^2}{2} + \frac{2|E|}{\sqrt{\xi}n}\cdot \frac{\xi n^2}{2}
=2|E|\sqrt{\xi}n
 \leq \frac{k}{1 + 3k} n|E|,
\end{split}
\end{equation}
where the third inequality follows from the concavity of the function $x^2$.
The rest of the proof is divided into two cases.\\

\noindent\textbf{Case 1}: $|E|\geq \frac{2 + 6k}{k} n^{-\frac13}\binom{n}{2}$. \\
Consider all triangles of $K_n$ with at least one edge in $E$. Note that if a triangle has at least one edge in $E$ and belongs to neither $\mathcal{T}_3(P)$ nor $\mathcal{T}_2(P)$, then it induces a rainbow triangle in $P$. 
Together with (\ref{eq: T2class}), we have 
\begin{equation*}
\begin{array}{ll}
\arraycolsep=1.4pt\def\arraystretch{1.2}
k |\mathcal{T}_3(P)| & \geq k\left(|E|(n-2) - 2|\mathcal{T}_2(P)| - 3n^{-\frac13}\binom{n}{3}\right)
  \geq k\left(\frac{1 + k}{1 + 3k}n|E| - 2|E| - 3n^{-\frac13}\binom{n}{3}\right)\\
 & = \frac{k}{1 + 3k}n|E| + k\left(\frac{k}{1 + 3k}n|E|- 2|E| - 3n^{-\frac13}\binom{n}{3}\right) \geq \frac{k}{1 + 3k}n|E|\geq |\mathcal{T}_2(P)|,\\
\end{array}
\end{equation*}
where the fourth inequality is given by $|E|\geq \frac{2 + 6k}{k} n^{-\frac13}\binom{n}{2}$ for sufficiently large $n$.
\\

\noindent\textbf{Case 2}: $|E|< \frac{2 + 6k}{k} n^{-\frac13}\binom{n}{2}$. \\
In this case, we have
\begin{equation*}
|\mathcal{T}_2(P)|< \frac12 |E|(n-2) < \frac{3 + 9k}{k}n^{-1/3}\binom{n}{3}.
\end{equation*}
\end{proof}

\subsection{Stability of Gallai $3$-templates of dense non-complete graphs}
\begin{lemma}\label{3graph: baltri}
Let $0<\xi \leq \frac{1}{64}$ and $n^{-1/3} \ll \delta \ll 1$. Let $G$ be a graph of order $n$, and $e(G)\geq (1-\xi)\binom{n}{2}$. Assume that $P$ is a Gallai $3$-template of $G$ with $|\mathrm{Ga}(P, G)|>2^{(1- \delta)\binom{n}{2}}$. Then $|\mathcal{T}_1(P)|\geq (1 - 40\delta)\binom{n}{3}$.
\end{lemma}

\begin{proof}
Let $|\mathcal{T}_1(P)|=\alpha\binom{n}{3}$, $|\mathcal{T}_2(P)|=\beta\binom{n}{3}$, $|\mathcal{T}_3(P)|=\eta\binom{n}{3}$ and $|\mathcal{T}_4(P)|=\gamma\binom{n}{3}$. Then $|\mathcal{T}_5(P)| \leq(1 - \alpha -\beta -\eta)\binom{n}{3}$.
Observe that for every $T\in \mathcal{T}_4(P)$, the template $P$ contains a rainbow triangle with edge set $T$; therefore, we obtain that $|\mathcal{T}_4(P)|\leq RT(P)\leq n^{-1/3}\binom{n}{3}$, which gives $\gamma \leq n^{-1/3}$.

Define for $e\in E(K_n)$ the weight function
\[
w(e)=
 \begin{cases}
    1 & \text{if $P(e)=\emptyset$}, \\
    |P(e)|  & \text{otherwise}. 
  \end{cases}
\]
Similarly to the proof of Lemma~\ref{baltri}, the number of Gallai $3$-colorings of $G$ which are subtemplates of $P$ satisfies
\begin{equation}\label{eq: common}
\arraycolsep=1.4pt\def\arraystretch{1.2}
\begin{array}{ll}
\log |\mathrm{Ga}(P, G)| & \leq \log\left(\prod_{e\in K_n}|w(e)|\right)
= \log \left(\prod_{T\in \mathcal{T}}\prod_{e\in T}|w(e)| \right)^{\frac{1}{n-2}} \\
& \leq \log \left(\prod_{T\in \mathcal{T}_1}2^3 \prod_{T\in \mathcal{T}_2}3^2 \prod_{T\in \mathcal{T}_3}6 \prod_{T\in \mathcal{T}_4}3^3\prod_{T\in \mathcal{T}_5}2^2 \right)\cdot \frac{1}{n-2}\\
&\leq \left(3\alpha +  2\beta\log 3 + \eta\log 6 + 3\gamma\log 3 + 2(1 - \alpha -\beta -\eta)\right)\frac13\binom{n}{2}\\
&=\left(2 + \alpha +  (2\log 3 - 2)\beta + (\log 6 - 2)\eta + 3n^{-1/3}\log 3\right)\frac13\binom{n}{2}.
\end{array}
\end{equation} 

Let $k=1$. By Lemma~\ref{small330}, we have $\beta\leq \max\{\eta,~12n^{-1/3}\}$. Assume that $\alpha< 1 - 40\delta$. The rest of the proof shall be divided into two cases.
\\

\noindent\textbf{Case 1}: $\beta \leq \eta$. \\
If $\eta < 20\delta$, continuing (\ref{eq: common}) we have 
\begin{equation*}
\arraycolsep=1.4pt\def\arraystretch{1.2}
\begin{array}{ll}
\log |\mathrm{Ga}(P, G)| &\leq \left(2 + \alpha +  \left(2\log 3 + \log 6 - 4\right)\eta + 3n^{-1/3}\log 3\right)\frac13\binom{n}{2}\\
& \leq \left(2 + (1 - 40\delta)+  1.8\cdot 20\delta + \delta\right)\frac13\binom{n}{2}
= (1 - \delta)\binom{n}{2}.
\end{array}
\end{equation*}
Otherwise, together with $\alpha\leq 1  - \beta - \eta$, continuing (\ref{eq: common}) we obtain that
\begin{equation*}
\arraycolsep=1.4pt\def\arraystretch{1.2}
\begin{array}{ll}
\log |\mathrm{Ga}(P, G)| 
&\leq \left(3+  (2\log 3 - 3)\beta + (\log 6 - 3)\eta + 3n^{-1/ 3}\log 3\right)\frac13\binom{n}{2}\\
& \leq \left(3 + \left(2\log 3 + \log 6 - 6\right)\eta + 3n^{-1/3}\log 3 \right)\frac13\binom{n}{2}\\
&\leq \left(3  - 0.2\cdot 20 \delta + \delta\right)\frac13\binom{n}{2}
= (1 - \delta)\binom{n}{2}.
\end{array}
\end{equation*} 

\noindent\textbf{Case 2}: $\beta\leq 12n^{-1/3}$. \\
Together with $\eta\leq 1 - \alpha$ and $\alpha< 1- 40\delta$, continuing (\ref{eq: common}) we have
\begin{equation*}
\arraycolsep=1.4pt\def\arraystretch{1.2}
\begin{array}{ll}
\log |\mathrm{Ga}(P, G)| & \leq \left(2 + \alpha + 2\log 3\cdot 12n^{-1/3} + (\log 6- 2)(1 - \alpha) + 3n^{-1/3}\log 3 \right)\frac13\binom{n}{2}\\
&\leq \left(\log 6 + (3- \log 6)\alpha + 27n^{-1/3}\log 3 \right)\frac13\binom{n}{2}\\
&\leq \left(\log 6 + (3 - \log 6)(1 - 40\delta)+ \delta\right)\frac13\binom{n}{2}
< (1 - \delta)\binom{n}{2}.
\end{array}
\end{equation*} 
Both cases contradict our assumption that $|\mathrm{Ga}(P, G)|>2^{(1- \delta)\binom{n}{2}}$.
\end{proof}

Similarly as in the proof of Theorem~\ref{2edge}, using  Lemmas~\ref{monoedge} and~\ref{3graph: baltri}, we obtain the following  theorem.
\begin{thm}\label{3graph: 2edge}
Let $0<\xi \leq \frac1{64}$ and $n^{-1/3} \ll \delta \ll 1$. Let $G$ be a graph of order $n$ and $e(G)\geq (1-\xi)\binom{n}{2}$. Assume that $P$ is a Gallai $3$-template of $G$ with $|\mathrm{Ga}(P, G)|>2^{(1- \delta)\binom{n}{2}}$. Then there exist two colors $i$, $j\in [3]$ such that the number of edges of $K_n$ with palette $\{i, j\}$ is at least $(1- 37\cdot 40\delta)\binom{n}{2}$.
\end{thm}

\subsection{Proof of Theorem~\ref{3graph}}
\noindent\textit{Proof of Theorem~\ref{3graph}.}
Let $\mathcal{C}$ be the collection of containers given by Theorem \ref{container} for $r=3$. Note that every Gallai $3$-coloring of $G$ is a subtemplate of some $P\in \mathcal{C}$.  
Let $\delta = \log^{-11} n$. We define
\[
\mathcal{C}_1 = \left\{P\in \mathcal{C} :\ |\mathrm{Ga}(P, K_n)|\leq 2^{(1- \delta)\binom{n}{2}}\right\},
\hspace{3mm}
\mathcal{C}_2 = \left\{P\in \mathcal{C} :\ |\mathrm{Ga}(P, K_n)|> 2^{(1- \delta)\binom{n}{2}}\right\}.
\]
Similarly to the proof of Theorem~\ref{rcomplete}, applying Theorems~\ref{container},~\ref{2colorclass}, and~\ref{3graph: 2edge}, we obtain that
\[
\begin{split}
|\mathrm{Ga}(\mathcal{C}, G)|&= \left|\mathrm{Ga}(\mathcal{C}_1, G)\right| + \left|\mathrm{Ga}(\mathcal{C}_2, G)\right|
\leq 2^{-\frac{n^2}{4\log^{11} n}}2^{\binom{n}{2}} + 3\cdot\left(2^{e(G)}+2^{-\frac{n}{3\log^2 n}}2^{\binom{n}{2}}\right)\\
& \leq 3\cdot 2^{e(G)} + 2^{-\frac{n}{4\log^2 n}}2^{\binom{n}{2}}.
\end{split}
\]
\qed

\section{Gallai $r$-colorings of non-complete graphs}
Theorem~\ref{rgraph} is a direct consequence of the following three theorems.
\begin{thm}\label{rgraph: high}
For $n, r\in \mathbb{N}$ with $r\geq 4$, there exists $n_0$ such that for all $n>n_0$ the following holds. For a graph $G$ of order $n$ with $e(G) \geq (1 - \log^{-11}n)\binom{n}{2}$, the number of Gallai $r$-colorings of $G$ is strictly less than $r^{\lfloor n^2/4 \rfloor}$.
\end{thm}

\begin{thm}\label{rgraph: low}
Let $n, r\in \mathbb{N}$ with $r\geq 4$, and $0<\xi\ll 1$. For a graph $G$ of order $n$ with $\lfloor n^2/4 \rfloor < e(G) \leq \lfloor n^2/4 \rfloor + \xi n^2$, the number of Gallai $r$-colorings of $G$ is strictly less than $r^{\lfloor n^2/4 \rfloor}$.
\end{thm}

\begin{thm}\label{rgraph: middle}
For $n, r\in \mathbb{N}$ with $r\geq 4$, there exists $n_0$ such that for all $n>n_0$ the following holds. Let $n^{-1/36}\ll \xi \leq \frac{1}{2}\log^{-11}n\ll 1$. For a graph $G$ of order $n$ with $(\frac14 + 3\xi)n^2\leq e(G) \leq (\frac12- 3\xi)n^2$, the number of Gallai $r$-colorings of $G$ is strictly less than $r^{\lfloor n^2/4 \rfloor}$.
\end{thm}

\subsection{Proof of Theorem~\ref{rgraph: high} for $r\geq 5$}
\begin{lemma}\label{lemma: rdense}
Let $n, r\in \mathbb{N}$ with $r\geq 5$ and $0<\xi \leq \frac 1{900}$. Assume that $G$ is a graph of order $n$ with $e(G)\geq (1-\xi)\binom{n}{2}$, and $P$ is a Gallai $r$-template of $G$. Then, for sufficiently large $n$, 
\[
|\mathrm{Ga}(P, G)| \leq r^{\frac12\binom{n}{2}}\cdot 2^{-0.007\binom{n}{2}}.
\]
\end{lemma}
\begin{proof}
Let $\mathcal{T}$ be the collection of triangles of $K_n$. For a given $r$-template $P$ of order $n$, we again use the partition~(\ref{eq: tripartition}). Let $|\mathcal{T}_1(P)|=\alpha\binom{n}{3}$, $|\mathcal{T}_2(P)|=\beta\binom{n}{3}$, $|\mathcal{T}_3(P)|=\eta\binom{n}{3}$ and $|\mathcal{T}_4(P)|=\gamma\binom{n}{3}$. Then $|\mathcal{T}_5(P)| \leq(1 - \alpha -\beta -\eta)\binom{n}{3}$. 
Note that for every $T\in \mathcal{T}_4(P)$, the template $P$ contains a rainbow triangle with edge set $T$; therefore, we obtain that $|\mathcal{T}_4(P)|\leq \mathrm{RT}(P)\leq n^{-1/3}\binom{n}{3}$, which gives $\gamma \leq n^{-1/3}$. 

Define for $e\in E(K_n)$ the weight function
\[
w(e)=
 \begin{cases}
    1 & \text{if $P(e)=\emptyset$} \\
    |P(e)|  & \text{otherwise}. 
  \end{cases}
\]
Similarly, as in Lemma~\ref{3graph: baltri}, the number of Gallai $r$-colorings of $G$, which is a subtemplate of $P$, satisfies
\begin{equation}\label{eq: common2}
\arraycolsep=1.4pt\def\arraystretch{1.2}
\begin{array}{ll}
\log|\text{Ga}(P, G)| 
& \leq \log\left(\prod_{T\in \mathcal{T}_1}2^3 \prod_{T\in \mathcal{T}_2}r^2 \prod_{T\in \mathcal{T}_3}2r \prod_{T\in \mathcal{T}_4}r^3 \prod_{T\in \mathcal{T}_5}2^2 \right)\cdot \frac{1}{n-2} \\
&\leq \left(3\alpha +  2\beta\log r + \eta\log 2r  + 3\gamma\log r  + 2(1 - \alpha -\beta -\eta)\right)\frac13\binom{n}{2}\\
&\leq \left(2 + \alpha +  (2\log r - 2)\beta + (\log r - 1)\eta + 3n^{-1/3}\log r \right)\frac13\binom{n}{2}.
\end{array}
\end{equation} 
Let $k=1/12$. By Lemma~\ref{small330}, we have $\beta\leq \max\{k\eta,~\frac{3 + 9k}{k}n^{-1/3}\}$. The rest of the proof shall be divided into two cases.
\\

\noindent\textbf{Case 1}: $\beta \leq k\eta$. \\
Together with $\alpha\leq (1 - \beta - \eta)$, continuing (\ref{eq: common2}) we have 
\begin{equation*}
\arraycolsep=1.4pt\def\arraystretch{1.2}
\begin{array}{ll}
\log |\mathrm{Ga}(P, G)| &\leq \left(3 +  (2\log r - 3)\beta + (\log r - 2)\eta + 3n^{-1/3}\log r \right)\frac13\binom{n}{2}\\
&\leq \left(3 +  \left((2k+1)\log r - (3k + 2)\right)\eta + 3n^{-1/3}\log r \right)\frac13\binom{n}{2}.
\end{array}
\end{equation*}
Note that $(2k+1)\log r - (3k + 2)$ is positive as $r\geq 4$. Therefore, together with $\eta\leq 1$ and $k=\frac{1}{12}$, we obtain that

\[\arraycolsep=1.4pt\def\arraystretch{1.2}
\begin{array}{ll}
\log|\mathrm{Ga}(P, G)| 
&\leq \left(\frac76\log r + \frac34 + 3n^{-1/3}\log r \right)\frac13\binom{n}{2}
\leq \left(\frac32\log r  - 0.023 + 3n^{-1/3}\log r \right)\frac13\binom{n}{2}\\
&\leq \frac12\binom{n}{2}\log r -0.007\binom{n}{2},
\end{array}
\]
where the second inequality follows from $(\frac13\log r - \frac34)\geq 0.023$ as $r\geq 5$.
\\

\noindent\textbf{Case 2}: $\beta\leq \frac{3 + 9k}{k}n^{-1/3}$. \\
Together with $\alpha\leq (1 - \eta)$, continuing (\ref{eq: common2}) we have
\[\arraycolsep=1.4pt\def\arraystretch{1.2}
\begin{array}{ll}
\log|\mathrm{Ga}(P, G)| 
& \leq \left(3 + (\log r - 2)\eta + 2\log r \cdot\frac{3 + 9k}{k}n^{-1/3} + 3n^{-1/3}\log r \right)\frac13\binom{n}{2}\\
& \leq \left(\frac32\log r - \left(\frac12\log r - 1\right) + \left(\frac{2 + 6k}{k} + 1\right)3n^{-1/3}\log r \right)\frac13\binom{n}{2}\\
&\leq \left(\frac32\log r - 0.16 + 0.01\right)\frac13\binom{n}{2}
= \frac12\binom{n}{2} \log r -0.05\binom{n}{2},
\end{array}
\]
where the third inequality holds for $r\geq 5$ and sufficiently large $n$.
\end{proof}

Using Lemma~\ref{lemma: rdense}, we prove a stronger theorem for the case $r\geq 5$.
\begin{thm}
For $n, r\in \mathbb{N}$ with $r\geq 5$ and $0<\xi \leq \frac 1{900}$, there exists $n_0$ such that for all $n>n_0$ the following holds. If $G$ is a graph of order $n$, and $e(G)\geq (1-\xi)\binom{n}{2}$, then the number of Gallai $r$-colorings of $G$ is less than $r^{\frac12\binom{n}{2}}.$
\end{thm}

\begin{proof}
Let $\mathcal{C}$ be the collection of containers given by Theorem \ref{container}. Theorem \ref{container} indicates that every Gallai $r$-coloring of $G$ is a subtemplate of some $P\in \mathcal{C}$ and $|\mathcal{C}|\leq 2^{cn^{-1/3}\log^2 n\binom{n}{2}}$ for some constant $c$, which only depends on $r$. 
We may assume that all templates $P$ in $\mathcal{C}$ are Gallai $r$-templates of $G$. By Property (ii) of Theorem~\ref{container}, we always have $\mathrm{RT}(P)\leq n^{-1/3}\binom{n}{3}$. Suppose that for a template $P$ there exists an edge $e\in E(G)$ with $|P(e)|=0$. Then we would obtain $|\mathrm{Ga}(P, G)| = 0$ as a Gallai $r$-coloring of $G$ requires at least one color on each edge. 
Now applying Lemma~\ref{lemma: rdense} on every container $P\in \mathcal{C}$, we obtain that the number of Gallai $r$-colorings of $G$ is at most
\[
\sum_{P\in \mathcal{C}}|\mathrm{Ga}(P, G)|\leq |\mathcal{C}| \cdot r^{\frac12\binom{n}{2}}\cdot 2^{-0.007\binom{n}{2}}< r^{\frac12\binom{n}{2}}
\]
for $n$ sufficiently large.
\end{proof}

\subsection{Proof of Theorem~\ref{rgraph: high} for $r=4$}
Given two colors $R$ and $B$, consider a $4$-template $P$ of order $n$ in which every edge of $K_n$ has palette $\{R, B\}$. 
For a constant $0<\varepsilon\ll 1$ and a graph $G$ with $e(G)> \binom{n}{2} - 2\varepsilon n$, we can easily check that $P$ is a Gallai $4$-template of $G$ and $|\mathrm{Ga}(P, G)|= 2^{e(G)}>4^{\frac12\binom{n}{2} - \varepsilon n}$. This indicates that Lemma~\ref{lemma: rdense} fails to hold when $r=4$. Instead, we shall apply the same technique as for $3$-colorings: prove a stability result to determine the approximate structure of $r$-templates, which would contain too many Gallai $r$-colorings, and then apply this together with Theorem~\ref{2colorclass} to obtain the desired bound.

\begin{lemma}\label{4graph: baltri}
Let $n^{-1/3} \ll \delta \ll 1$. Let $G$ be a graph of order $n$ with $e(G)\geq (1-\delta)\binom{n}{2}$. Assume that $P$ is a Gallai $4$-template of $G$ with $|\mathrm{Ga}(P, G)|>2^{(1- \delta)\binom{n}{2}}$. Then the number of triangles $T$ of $K_n$ with $\sum_{e\in T}|P(e)|=6$ and $P(e)=P(e')$ for every $e$, $e'\in T$ is at least $(1 - 16\delta)\binom{n}{3}$.
\end{lemma}

\begin{proof}
Let $\mathcal{T}$ be the collection of triangles of $K_n$. We define
\begin{equation*}
\arraycolsep=1.4pt\def\arraystretch{1.2}
\begin{array}{ll}
\mathcal{T}_1&=\left\{T \in \mathcal{T} \mid \textstyle\sum_{e\in T}|P(e)|=6 \text{ and } P(e)=P(e')\text{ for every }e,~e'\in T\right\},\\
\mathcal{T}_2&=\left\{T \in \mathcal{T}\mid \exists~e\in T, \ |P(e)|=0\right\},\\
\mathcal{T}_3&=\left\{T \in \mathcal{T}\mid T=\{e_1, e_2, e_3\},~|P(e_1)|= 4,~|P(e_2)|=|P(e_3)|=1\right\},\\
\mathcal{T}_4&=\left\{T \in \mathcal{T}\setminus (\mathcal{T}_1 \cup \mathcal{T}_2\cup \mathcal{T}_3) \mid \textstyle\sum_{e\in T}|P(e)| \geq 6\right\},\\
\mathcal{T}_5&=\left\{T \in \mathcal{T}\setminus \mathcal{T}_2 \mid \textstyle\sum_{e\in T}|P(e)|\leq 5\right\}.
\end{array}
\end{equation*}

Let $|\mathcal{T}_1|=\alpha\binom{n}{3}$, $|\mathcal{T}_2|=\beta\binom{n}{3}$, $|\mathcal{T}_3|=\eta\binom{n}{3}$ and $|\mathcal{T}_4|=\gamma\binom{n}{3}$. Then $|\mathcal{T}_5| =(1 - \alpha -\beta -\eta - \gamma)\binom{n}{3}$.
Since $G$ satisfies $e(G)\geq (1 - \delta)\binom{n}{2}$ and $P$ is a Gallai template, we have $|\mathcal{T}_2|\leq \delta\binom{n}{2}\cdot n\leq 6\delta\binom{n}{3}$, and therefore $\beta\leq 6\delta$.
Observe that for every $T\in \mathcal{T}_4$, the template $P$ contains a rainbow triangle with edge set $T$; therefore, we obtain that $|\mathcal{T}_4|\leq RT(P)\leq n^{-1/3}\binom{n}{3}$, which gives $\gamma \leq n^{-1/3}$.

Define for $e\in E(K_n)$ the weight function
\[
w(e)=
 \begin{cases}
    1 & \text{if $P(e)=\emptyset$} \\
    |P(e)|  & \text{otherwise}. 
  \end{cases}
\]
Assume that $\alpha< 1 - 16\delta$. Similarly, as in Lemma~\ref{3graph: baltri}, the number of Gallai $4$-colorings of $G$ which is a subtemplate of $P$ satisfies
\begin{equation*}
\arraycolsep=1.4pt\def\arraystretch{1.2}
\begin{array}{ll}
\log|\mathrm{Ga}(P, G)| 
& \leq \log \left(\prod_{T\in \mathcal{T}_1}2^3 \prod_{T\in \mathcal{T}_2}4^2 \prod_{T\in \mathcal{T}_3}4 \prod_{T\in \mathcal{T}_4}4^3 \prod_{T\in \mathcal{T}_4}4 \right)\cdot \frac{1}{n-2}\\
&\leq \left(3\alpha +  4\beta + 2\eta + 6\gamma + 2(1 - \alpha -\beta -\eta - \gamma)\right)\frac13\binom{n}{2}\\
&= \left(2 +\alpha +  2\beta + 4\gamma \right)\frac13\binom{n}{2}
< \left(2 + (1 - 16\delta) +  13\delta\right)\frac13\binom{n}{2}
=(1 - \delta)\binom{n}{2}.
\end{array}
\end{equation*} 
This contradicts the assumption that $|\mathrm{Ga}(P, G)|>2^{(1- \delta)\binom{n}{2}}$.
\end{proof}

Similarly, as in Theorem~\ref{2edge}, applying  Lemmas~\ref{monoedge} and~\ref{4graph: baltri}, we obtain the following. 

\begin{thm}\label{4graph: 2edge}
Let $n^{-1/3} \ll \delta \ll 1$. Let $G$ be a graph of order $n$ with $e(G)\geq (1-\delta)\binom{n}{2}$. Assume that $P$ is a Gallai $4$-template of $G$ with $|\mathrm{Ga}(P, G)|>2^{(1- \delta)\binom{n}{2}}$. Then there exist two colors $i$, $j\in [4]$ such that the number of edges of $K_n$ with palette $\{i, j\}$ is at least $(1- 145\cdot 16\delta)\binom{n}{2}$.
\end{thm}

\noindent\textit{Proof of Theorem~\ref{rgraph: high} for $r=4$.}
Let $\mathcal{C}$ be the collection of containers given by Theorem~\ref{container} for $r=4$. Note that every Gallai $4$-coloring of $G$ is a subtemplate of some $P\in \mathcal{C}$.  
Let $\delta = \log^{-11} n$. We define
\[
\mathcal{C}_1 = \left\{P\in \mathcal{C} :\ |\mathrm{Ga}(P, G)|\leq 2^{(1- \delta)\binom{n}{2}}\right\},
\hspace{3mm}
\mathcal{C}_2 = \left\{P\in \mathcal{C} :\ |\mathrm{Ga}(P, G)|> 2^{(1- \delta)\binom{n}{2}}\right\}.
\]
Similarly, as in the proof of Theorem~\ref{rcomplete}, applying Theorems~\ref{container},~\ref{2colorclass}, and~\ref{4graph: 2edge}, we obtain that
\[
\begin{split}
|\mathrm{Ga}(\mathcal{C}, G)|&= \left|\mathrm{Ga}(\mathcal{C}_1, G)\right| + \left|\mathrm{Ga}(\mathcal{C}_2, G)\right|
\leq 2^{-\frac{n^2}{4\log^{11} n}}2^{\binom{n}{2}} + 6\left(2^{e(G)}+2^{-\frac{n}{3\log^2 n}}2^{\binom{n}{2}}\right)\\
& \leq 6\cdot 2^{e(G)} + 2^{-\frac{n}{4\log^2 n}}2^{\binom{n}{2}}
<4^{\lfloor n^2/4 \rfloor}.
\end{split}
\]

\subsection{Proof of Theorem~\ref{rgraph: low}}
A \textit{book} of size $q$ consists of $q$ triangles sharing a common edge, which is known as the \textit{base} of the book.
We write $\mathrm{bk} (G)$ for the size of the largest book in a graph $G$ and call it the \textit{booksize} of $G$. 

\begin{lemma}\label{lower: l1}
Let $n, r\in \mathbb{Z}^+$ with $r\geq 4$, $0<\alpha, \beta\ll 1$, and $G$ be a graph of order $n$. Assume that there exists a partition $V(G)=A\cup B$ satisfying the following conditions:
\begin{enumerate}[label={\upshape(\roman*)}]
\item $\delta(G[A, B])\geq (\frac12 - \alpha)n$;
\item $\Delta(G[A])$, $\Delta(G[B])\leq \beta n$.
\end{enumerate} 
Then the number of Gallai $r$-colorings of $G$ is at most $r^{\lfloor n^2/4\rfloor}$. 
Furthermore, if $e(G)\neq\lfloor n^2/4\rfloor$, then the number of Gallai $r$-colorings of $G$ is strictly less than $r^{\lfloor n^2/4\rfloor}$.
\end{lemma}
\begin{proof}
By Condition~(i), we have $(\frac12 - \alpha)n \leq |A|, |B|\leq (\frac12 + \alpha)n$.
Let $e(G)=\lfloor n^2/4\rfloor + m$. 
Without loss of generality, we can assume that $m>0$ and $e(G[A])\geq \frac{m}{2}$. Then there exists a matching $M$ in $G[A]$ of size at least $\frac{e(G[A])}{2\Delta(G[A])-1}\geq \frac{m}{4\beta n}$.

For two vertices $u, v\in A$, the number of their common neighbors in $B$ is at least
\[
|B| - 2\left(|B| - \delta(G[A, B])\right)= 2\delta(G[A, B]) - |B|\geq 2\left(\frac12 - \alpha\right)n - \left(\frac12 + \alpha\right)n\geq \frac{n}{3}.
\]
Then, for every $e\in G[A]$, there exists a book graph $B_{e}$ of size $n/3$ with the base $e$. 
Let $\mathcal{B}=\{B_{e} \mid e \in M\}$. Note that $M$ is a matching, and therefore book graphs in $\mathcal{B}$ are edge-disjoint. 
Another crucial fact is that for every $B\in \mathcal{B}$, the number of $r$-colorings of $B$ without rainbow triangles is at most $r\left(r + 2(r-1)\right)^{n/3}<r(3r)^{n/3}$, since once we color the base edge, each triangle must be colored in the way that two of its edges share the same color.
Hence, the number of Gallai $r$-colorings of $G$ is at most
\[
\left(r(3r)^{\frac n 3}\right)^{|M|}r^{e(G) - |M|\left(1 + 2\cdot \frac n3\right)}
= r^{e(G) - (1 - \log_r 3)|M|\cdot \frac n3}
\leq r^{\lfloor n^2/4 \rfloor + m - (1 - \log_r 3)\frac{m}{4\beta n}\cdot \frac n3}
< r^{\lfloor n^2/4 \rfloor},
\]
where the last inequality is given by $\beta\ll 1$.
\end{proof}

\begin{lemma}\label{lower: l2}
Let $n, r\in \mathbb{Z}^+$ with $r\geq 4$, $0<\alpha', \beta\ll 1$, $0<\alpha, \gamma, \xi\ll \varepsilon\ll 1$, and $G$ be a graph of order $n$ with $e(G)\leq \lfloor n^2/4\rfloor + \xi n^2$. Assume that there exists a partition $V(G)=A\cup B\cup C$ satisfying the following conditions:
\begin{enumerate}[label={\upshape(\roman*)}]
\item $d_{G[A, B]}(v)\geq \left(\frac12 - \alpha\right)n$ for all but at most $\gamma n$ vertices in $A\cup B$;
\item $\delta(G[A, B])\geq \left(\frac12 - \alpha'\right)n$;
\item $\Delta(G[A])$, $\Delta(G[B])\leq \beta n$;
\item $0<|C|\leq \gamma n$;
\item for every $v\in C$, both $d(v, A)$, $d(v, B)\geq r\varepsilon n$.
\end{enumerate} 
Then the number of Gallai $r$-colorings of $G$ is strictly less than $r^{\lfloor n^2/4\rfloor}$.
\end{lemma}
\begin{proof}
By Condition~(i), we have 
\begin{equation}\label{eq: partsize}
\left(\frac12 - \alpha\right)n \leq |A|, |B|\leq \left(\frac12 + \alpha\right)n.
\end{equation}
For a vertex $v$, a set $S$, a set of colors $\mathcal{R}$ and a coloring of $G$, let $N(v, S; \mathcal{R})$ be the set of vertices $u\in N(v, S)$, such that $uv$ is colored by some color in $\mathcal{R}$. Let $d(v, S; \mathcal{R})=|N(v, S; \mathcal{R})|$.
Denote by $\mathcal{C}_1$ the set of Gallai $r$-colorings of $G$, in which there exist a vertex $v\in C$, and two disjoint sets of colors $\mathcal{R}_1$ and $\mathcal{R}_2$, such that both $d(v, A; \mathcal{R}_1)$, $d(v, B; \mathcal{R}_2)\geq \varepsilon n$. Let $\mathcal{C}_2$ be the set of Gallai $r$-colorings of $G$, which are not in $\mathcal{C}_1$.

We first show that $\mathcal{C}_1= o(r^{\lfloor n^2/4\rfloor})$. We shall count the ways to color $G$ so that the resulting colorings are in $\mathcal{C}_1$. First, we color the edges in $G[C, A\cup B]$; the number of ways is at most $r^{e(G[C, A\cup B])}$. 
Once we fix the colors of edges in $G[C, A\cup B]$, by the definition of $\mathcal{C}_1$, there exist a vertex $v\in C$, and two disjoint sets of colors $\mathcal{R}_1$ and $\mathcal{R}_2$, such that $d(v, A; \mathcal{R}_1)$, $d(v, B; \mathcal{R}_2)\geq \varepsilon n$. 
We observe that for every edge $e=uw$ between $N_1=N(v, A; \mathcal{R}_1)$ and $N_2=N(v, B; \mathcal{R}_2)$, $e$ either shares the same color with $uv$, or with $vw$, as otherwise we would obtain a rainbow triangle $uvw$. Then the number of ways to color edges in $G[N_1, N_2]$ is at most $2^{e(G[N_1, N_2])}\leq r^{\frac12 e(G[N_1, N_2])}$. 
Note that by Condition~(i), inequality~(\ref{eq: partsize}) and $\alpha, \gamma\ll \varepsilon$, we have
\[
e(G[N_1, N_2])\geq (|N_1| - \gamma n) (|N_2| - 2\alpha n)\geq \frac12 \varepsilon^2 n^2.
\]
Hence, we obtain
\[
\begin{split}
\log_r|\mathcal{C}_1| & \leq e(G[C, A\cup B]) + \frac12e(G[N_1, N_2]) + (e(G)- e(G[C, A\cup B])- e(G[N_1, N_2]))\\
&= e(G)- \frac12 e(G[N_1, N_2])
\leq \lfloor n^2/4\rfloor + \xi n^2 - \frac14 \varepsilon^2 n^2,
\end{split}
\]
which indicates $|\mathcal{C}_1|= o(r^{\lfloor n^2/4\rfloor})$ as $\xi \ll \varepsilon$.

It remains to estimate the size of $\mathcal{C}_2$. Recall that for a coloring in $\mathcal{C}_2$, for every vertex $v\in C$, there are no two disjoint sets of colors $\mathcal{R}_1$ and $\mathcal{R}_2$ such that $d(v, A; \mathcal{R}_1)$, $d(v, B; \mathcal{R}_2)\geq \varepsilon n$. 
\begin{claim}\label{claim: dominatecolor}
Let $\mathcal{S}$ be a set of $r$ colors. For every coloring in $\mathcal{C}_2$, and every vertex $v\in C$, there exists a color $R\in \mathcal{S}$, such that both $d(v, A; \mathcal{S}\setminus\{R\})<\varepsilon n$ and $d(v, B;  \mathcal{S}\setminus\{R\})<\varepsilon n$. 
\end{claim}
\begin{proof}
We arbitrarily fix a coloring in $\mathcal{C}_2$ and a vertex $v\in C$. By Condition~(v), there exists a color $R$ such that $d(v, A; R)\geq \varepsilon n$. By the definition of $\mathcal{C}_2$, we obtain that $d(v, B; \mathcal{S}\setminus\{R\})<\varepsilon n$. Then we also have $d(v, B; R)\geq d(v, B)-d(v, B; \mathcal{S}\setminus\{R\})\geq r\varepsilon n - \varepsilon n >\varepsilon n$. For the same reason, we obtain that $d(v, A; \mathcal{S}\setminus\{R\})<\varepsilon n$.
\end{proof}
By Claim~\ref{claim: dominatecolor}, the number of ways to color edges in $G[C, A\cup B]$ is at most
\[
\left(r\sum_{i\leq \varepsilon n}\binom{n}{i}\sum_{i\leq \varepsilon n}\binom{n}{i} r^{2\varepsilon n}\right)^{|C|}
\leq\left(4r\left(\frac{ne}{\epsilon n}\right)^{2\varepsilon n} r^{2\varepsilon n}\right)^{|C|}
\leq r^{\left(\left(\log_r e - \log_r \varepsilon +1\right)2\varepsilon n + 2\right)|C|}
< r^{\frac{|C|n}{3}},
\]
where the last inequality is given by $\left(\log_r e - \log_r \varepsilon +1\right)2\varepsilon\ll \frac13$ as $\varepsilon \ll 1$.
Note that by Conditions~(ii)--(iv), we have 
\begin{itemize}
 \setlength\itemsep{-0.5ex}
\item $\delta(G[A, B])\geq  \left(\frac12 - \alpha'\right)n\geq \left(\frac12 - \alpha'\right)(|A| + |B|)$;
\item $\Delta(G[A]), \Delta(G[B])\leq \beta n \leq \frac{\beta}{1 - \gamma}(|A| + |B|)$.
\end{itemize}
Applying Lemma~\ref{lower: l1} on $G[A\cup B]$, we obtain that the number of ways to color edges in $G[A\cup B]$ is at most $r^{\frac{(n - |C|)^2}{4}}$. A trivial upper bound for the ways to color the rest of the edges, that is, the edges in $G[C]$ is  $r^{\binom{|C|}{2}}$.
Hence, we have
\[
\log_r |\mathcal{C}_2|\leq \frac{|C|n}{3} + \frac{(n - |C|)^2}{4} + \binom{|C|}{2}
= \frac{n^2}{4} - \left(\frac{n}{6} - \frac34|C| + \frac12\right)|C|\leq \lfloor n^2/4\rfloor - \frac14,
\]
where the last inequality is given by $0<|C|\leq \gamma n$ and $\gamma \ll1$. 
Finally, we obtain that the number of Gallai $r$-colorings of $G$ is
\[
|\mathcal{C}_1| + |\mathcal{C}_2| \leq o(r^{\lfloor n^2/4\rfloor}) + r^{\lfloor n^2/4\rfloor - \frac14} < r^{\lfloor n^2/4\rfloor}.
\]
\end{proof}

\begin{lemma}\label{lower: l3}
Let $n, r\in \mathbb{Z}^+$ with $r\geq 4$, $\alpha, \beta, \gamma, \xi\ll 1$, and $G$ be a graph of order $n$ with $\lfloor n^2/4\rfloor < e(G)\leq \lfloor n^2/4\rfloor 
+ \xi n^2$. Assume, that there exists a partition $V(G)=A\cup B\cup C$ satisfying the following conditions:
\begin{enumerate}[label={\upshape(\roman*)}]
\item $\delta(G[A, B])\geq (\frac12 - \alpha)n$;
\item $\Delta(G[A])$, $\Delta(G[B])\leq \beta n$;
\item $0<|C|\leq \gamma n$;
\item for every $v\in C$, $d(v)\geq n/2$.
\end{enumerate} 
Then the number of Gallai $r$-colorings of $G$ is strictly less than $r^{\lfloor n^2/4\rfloor}$.
\end{lemma}
\begin{proof}
Let $\alpha, \gamma, \xi \ll \varepsilon \ll 1$. Let $C_1=\{v\in C \mid d(v, A)< r\varepsilon n\}$, and $C_2=\{v\in C \mid d(v, B)< r\varepsilon n\}$. By Conditions~(iii) and (iv), for every $v\in C_1$, we have $d(v, B)\geq \left(\frac{1}{2} -\gamma - r\varepsilon\right) n$. Similarly, for every $v\in C_2$, we have $d(v, A)\geq \left(\frac{1}{2} - \gamma - r\varepsilon\right) n$. Define 
\[
A'=A \cup C_1, \quad B'=B \cup C_2, \quad C'=C \setminus(C_1\cup C_2).
\]
If $C'=\emptyset$, then we obtain a new partition $V(G)=A'\cup B'$ satisfying the following properties:
\begin{itemize}
 \setlength\itemsep{-0.5ex}
\item $\delta(G[A', B'])\geq \min\{\left(\frac12 - \alpha\right)n, \left(\frac{1}{2} - \gamma - r\varepsilon\right) n\}
=\left(\frac{1}{2}  -\gamma - r\varepsilon\right) n$;
\item $\Delta(G[A'])$, $\Delta(G[B'])\leq \min\{(\beta + \gamma)n, (r\varepsilon + \gamma)n\}$.
\end{itemize}
Together with $e(G)>\lfloor n^2/4\rfloor$, by Lemma~\ref{lower: l1}, we obtain that the number of Gallai $r$-colorings of $G$ is strictly less than $r^{\lfloor n^2/4\rfloor}$.
Otherwise, we obtain a new partition $V(G)=A'\cup B'\cup C'$ satisfying the following properties:
\begin{itemize}
  \setlength\itemsep{-0.5ex}
\item $d_{G[A', B']}(v)\geq \left(\frac12 - \alpha\right)n$ for all but at most $\gamma n$ vertices in $A'\cup B'$;
\item $\delta(G[A', B'])\geq \left(\frac{1}{2}  - \gamma - r\varepsilon\right) n$;
\item $\Delta(G[A'])$, $\Delta(G[B'])\leq \min\{(\beta + \gamma)n, (r\varepsilon + \gamma)n\}$;
\item $0<|C'|\leq |C|\leq \gamma n$;
\item for every $v\in C'$, both $d(v, A')$, $d(v, B')\geq r\varepsilon n$.
\end{itemize}
Together with $e(G)\leq \lfloor n^2/4\rfloor + \xi n^2$, by Lemma~\ref{lower: l2}, the number of Gallai $r$-colorings of $G$ is strictly less than $r^{\lfloor n^2/4\rfloor}$.
\end{proof}

Now, we prove a lemma which is crucial to the proof of Theorem~\ref{rgraph: low}.
\begin{lemma}\label{lower: l4}
Let $n, r\in \mathbb{Z}^+$ with $r\geq 4$, $\alpha, \beta, \gamma, \xi\ll 1$, and $G$ be a graph of order $n$ with $\lfloor n^2/4\rfloor < e(G)\leq \lfloor n^2/4\rfloor + \xi n^2$. Assume that there exists a partition $V(G)=A\cup B\cup C$ satisfying the following conditions:
\begin{enumerate}[label={\upshape(\roman*)}]
\item $\delta(G[A, B])\geq (\frac12 - \alpha)n$;
\item $\Delta(G[A])$, $\Delta(G[B])\leq \beta n$;
\item $|C|\leq \gamma n$.
\end{enumerate} 
Then the number of Gallai $r$-colorings of $G$ is strictly less than $r^{\lfloor n^2/4\rfloor}$.
\end{lemma}
\begin{proof}
By Lemma~\ref{lower: l1}, we can assume that $|C|>0$ without loss of generality. We begin with the graph $G$, greedily remove a vertex in $C$ with degree strictly less than $|G|/2$ in $G$ to obtain a smaller subgraph. Let $G'$ be the resulting graph when the algorithm terminates, and $n'=|V(G')|$. We remark that $G'$ is not unique and it depends on the order of removing vertices.
Without loss of generality, we can assume that $n'<n$, as otherwise we are done by applying Lemma~\ref{lower: l3} on $G$.

Let $A'=A$, $B'=B$, and $C'=V(G')\cap C$. Clearly, we have $G'=G[A'\cup B'\cup C']$. 
Furthermore, by the mechanics of the algorithm, we have
\begin{equation}\label{l4: Gedge}
e(G)\leq e(G') + \frac12\left(\binom{n}{2} - \binom{n'}{2}\right).
\end{equation}
We first claim that $e(G')> \lfloor (n')^2/4\rfloor$, as otherwise we would have 
\[
e(G)
\leq \lfloor (n')^2/4\rfloor + \frac12\left(\binom{n}{2} - \binom{n'}{2}\right)
\leq \lfloor n^2/4\rfloor,
\]
which contradicts the assumption of the lemma. 
On the other hand, since $n'\geq (1- \gamma)n$, we obtain that
\[
e(G)\leq \lfloor n^2/4\rfloor + \xi n^2\leq \lfloor (n')^2/4\rfloor + \frac{\gamma + 2\xi}{2(1 - \gamma)^2}(n')^2.
\]
Let $\xi'=\frac{\gamma + 2\xi}{2(1 - \gamma)^2}$. Then we have
\begin{equation}\label{l4: edge}
\lfloor (n')^2/4\rfloor <e(G')\leq \lfloor (n')^2/4\rfloor + \xi '(n')^2.
\end{equation}
If $C'=\emptyset$, we obtain a vertex partition $V(G')=A'\cup B'$ satisfying:
\begin{itemize}
  \setlength\itemsep{-0.5ex}
\item $\delta(G'[A', B'])\geq (\frac12 - \alpha)n\geq (\frac12 - \alpha)n'$;
\item $\Delta(G'[A])$, $\Delta(G'[B])\leq \beta n\leq \frac{\beta}{1 - \gamma}n'$.
\end{itemize} 
Together with~(\ref{l4: edge}), by Lemma~\ref{lower: l1}, we obtain that the number of Gallai $r$-colorings of $G'$, denoted by $|\mathcal{C}(G')|$, is strictly less than $r^{\lfloor n^2/4\rfloor}$.
Otherwise, we find the partition $V(G')=A'\cup B' \cup C'$ satisfying:
\begin{itemize}
  \setlength\itemsep{-0.5ex}
\item $\delta(G'[A', B'])\geq (\frac12 - \alpha)n\geq (\frac12 - \alpha)n'$;
\item $\Delta(G'[A])$, $\Delta(G'[B])\leq \frac{\beta}{1 - \gamma}n'$;
\item $0<|C'|\leq \gamma n\leq \frac{\gamma}{1 - \gamma}n'$;
\item for every $v\in C'$, $d(v)\geq \frac{n'}{2}$.
\end{itemize} 
Together with~(\ref{l4: edge}), by Lemma~\ref{lower: l3}, we obtain that $|\mathcal{C}(G')|<r^{\lfloor (n')^2/4\rfloor}$. Combining with~(\ref{l4: Gedge}), we conclude that the number of Gallai $r$-colorings of $G$, denoted by $|\mathcal{C}(G)|$, satisfies
\[
\log_r|\mathcal{C}(G)|\leq \log_r|\mathcal{C}(G')| + (e(G) - e(G'))
< \lfloor (n')^2/4\rfloor + \frac{1}{2}\left(\binom{n}{2} - \binom{n'}{2}\right)
\leq \lfloor n^2/4\rfloor,
\]
which completes the proof.
\end{proof}

Another important tool we need is the stability property of book graphs proved by Bollob{\'a}s and Nikiforov~\cite{BN}. 
\begin{thm}\label{bookstable}\cite{BN}
For every $0 < \alpha < 10^{-5}$ and every graph $G$ of order $n$ with $e(G)\geq (\frac14 - \alpha)n^2$, either
\[
\mathrm{bk}(G)>\left(\frac16 - 2\alpha^{1/3}\right)n
\]
or $G$ contains an induced bipartite graph $G_1$ of order at least $(1 - \alpha^{1/3})n$ and with minimum degree
\[
\delta(G_1)\geq \left(\frac12 - 4\alpha^{1/3}\right)n.
\]
\end{thm}

\noindent\textit{Proof of Theorem~\ref{rgraph: low}: } Let $e(G)= \lfloor n^2/4 \rfloor + m$, where $0<m\leq \xi n^2$. We construct a family $\mathcal{B}$ of book graphs by the following algorithm. We start the algorithm with $\mathcal{B}=\emptyset$ and $G_0=G$. In the i-th iteration step, if there exists a book graph $B$ of size $\frac{n}{7}$ in $G_i$, we let $\mathcal{B}=\mathcal{B}\cup \{B\}$, and $G_{i}=G_{i-1}-e$, where $e$ is the base edge of $B$. The algorithm terminates when there is no book graph of size $n/7$. Let $E_0$ be the set of base edges of $\mathcal{B}$, and $\tau=7/(1 - \log_r 3)$.

Suppose that $|\mathcal{B}|\geq 2\tau m$. 
Since $|E_0|=|\mathcal{B}|\geq 2\tau m$, the edge set $E_0$ contains a matching $M$ of size $\frac{|E_0|}{2(n-1) -1}> \tau m/n$. 
Let $\mathcal{B}'$ be the set of book graphs in $\mathcal{B}$ whose base edges are in $M$. Since $M$ is a matching, book graphs in $\mathcal{B}'$ are edge-disjoint. 
Note that for every $B\in \mathcal{B}‘$, the number of $r$-colorings of $B$ without rainbow triangles is at most $r(r + 2(r-1))^{n/7}<r(3r)^{n/7}$. Then the number of Gallai colorings of $G$ is at most
\[
\left(r(3r)^{\frac n 7}\right)^{|M|}r^{e(G) - |M|(1 + 2\cdot \frac n7)}
=r^{\lfloor n^2/4 \rfloor + m - (1 - \log_r 3)|M|\frac n7}
< r^{\lfloor n^2/4 \rfloor + m - m}=r^{\lfloor n^2/4\rfloor}.
\]

It remains to consider the case for $|\mathcal{B}|< 2\tau m$. 
Without loss of generality, we can assume that there is no matching of size greater than $\tau m/n$ in $E_0$.
Let $G'=G - E_0$. Then we have
\[
e(G') > \lfloor n^2/4 \rfloor - \left(2\tau - 1\right) m.
\]
Furthermore, by the construction of $G'$, we obtain that $\mathrm{bk}(G')<n/7$. 
Let $\alpha = \left(2\tau - 1\right) \xi.$
By applying Theorem~\ref{bookstable} on $G'$, we obtain that there is a vertex partition $V(G')=A'\cup B'\cup C'$ with $|C'|\leq \alpha^{1/3}n$, such that $A', B'$ are independent sets, and
\[
\delta(G'[A', B'])\geq \left(\frac12 - 4\alpha^{1/3}\right)n.
\]
Let $G_0$ be the spanning subgraph of $G$ with edge set $E_0$. For a small constant $\beta$ with $\xi\ll \beta\ll1$, let $V_0$ be the set of vertices in $G_0$ with degree more than $\beta n$. Since $|E_0|< 2\tau m\leq 2\tau \xi n^2$, we have $|V_0|\leq (4\tau\xi/\beta) n\leq \beta n$. Let $A=A'\setminus V_0$, $B=B'\setminus V_0$, and $C=C'\cup V_0$. Then we obtain a vertex partition $V(G)=A\cup B\cup C$ satisfying the following conditions:
\begin{itemize}
  \setlength\itemsep{-0.5ex}
\item $\delta(G[A, B])\geq (\frac12 - 4\alpha^{1/3} - \beta)n$;
\item $\Delta(G[A])$, $\Delta(G[B])\leq \beta n$;
\item $|C|\leq (\alpha^{1/3} + \beta)n.$
\end{itemize}
By Lemma~\ref{lower: l4}, we obtain that the number of Gallai $r$-colorings of $G$ is strictly less than $r^{\lfloor n^2/4\rfloor}$.
\qed

\subsection{Proof of Theorem~\ref{rgraph: middle}}
We say that a graph $G$ is \textit{$t$-far} from being $k$-partite if $\chi(G')>k$ for every subgraph $G'\subset G$ with $e(G') >  e(G) - t$. We will use the following theorem of Balogh, Bushaw, Collares, Liu, Morris, and Sharifzadeh~\cite{balogh2017typical}.
\begin{thm}\label{thm: tfar}\cite{balogh2017typical}
For every $n, k, t \in \mathbb{N}$, the following holds. Every graph $G$ of order $n$ which is $t$-far from being $k$-partite contains at least
\[
\frac{n^{k-1}}{e^{2k}\cdot k!}\left(e(G) + t - \left(1 - \frac 1 k \right) \frac{n^2}{2}\right)
\]
copies of $K_{k+1}$.
\end{thm}

\begin{prop}\label{prop: edge}
Let $n \in \mathbb{N}$ and $0<\varepsilon\leq 1$. Every graph $F$ on at least $\varepsilon n$ vertices, which contains at most $n^{-1/3}\binom{n}{3}$ triangles, satisfies
\[
e(F)\leq \frac{|F|^2}{4} + \frac{e^4}{6n^{1/3}\varepsilon^3}|F|^2.
\]
\end{prop}
\begin{proof}
Let $t= \frac{e^4}{6n^{1/3}\varepsilon^3}|F|^2$. Assume that $e(F) > \frac{|F|^2}{4} + t$. Then $F$ is $t$-far from being bipartite. By Theorem~\ref{thm: tfar}, the number of triangles in $F$ is at least 
\[
\frac{|F|}{2e^{4}}\left(e(F) + t - \frac{|F|^2}{4}\right)>\frac{|F|}{2e^{4}} \cdot 2t = \frac{1}{6n^{1/3}\varepsilon^3}|F|^3>  n^{-1/3}\binom{n}{3},
\]
which gives a contradiction.
\end{proof}

For an $r$-template $P$ of order $n$, we say that an edge $e$ of $K_n$ is an \textit{$r$-edge} of $P$ if $|P(e)|\geq 3$. An $r$-edge $e$ is \textit{typical} if the number of rainbow triangles containing $e$ is at most $n^{11/12}$. We then immediately obtain the following proposition.

\begin{prop}\label{prop: redge}
For an $r$-template of order $n$ containing at most $n^{-1/3}\binom{n}{3}$ rainbow triangles, the number of $r$-edges of $P$, which is not typical, is at most $n^{11/6}$.
\end{prop}

We now prove the following lemma.
\begin{lemma}\label{lemma: middle}
Let $n, r\in \mathbb{N}$ with $r\geq 4$, and $n^{-1/33}\ll \xi \leq \frac{1}{2}\log^{-11}n\ll 1$. Assume that $G$ is a graph of order $n$ with $(\frac14 + 3\xi)n^2\leq e(G) \leq (\frac12- 3\xi)n^2$, and $P$ is a Gallai $r$-template of $G$. Then, for sufficiently large $n$, 
\[
\log_r|\mathrm{Ga}(P, G)| \leq \frac{n^2}{4} - \xi^3\frac{n^2}{2} + 4n^{23/12}.
\]
\end{lemma}

\begin{proof}
We first construct a subset $I$ of $[n]$ and a sequence of graphs $\{G_0, G_1, \ldots, G_{\ell}\}$ by the following algorithm. We start the algorithm with $I = \emptyset$ and $G_0 = G$. In the $i$-th iteration step, we either add a vertex $v$ to $I$, whose degree is at most $(\frac12 - \xi^2)(|G_i|-1)$ in the graph $G_i$, or add a pair of vertices $\{u, v\}$ to $I$, where $uv$ is a typical $r$-edge satisfying $|N_{G_i}(u) \cap N_{G_i}(v)|\geq 2\xi^2(|G_i| - 2)$. 
In both cases, we define $G_{i+1}= G - I$. The algorithm terminates when neither of the above types of vertices exists. 
 
Assume that the algorithm terminates after $\ell$ steps. Let $G' = G_{\ell}$ and $k=|G'|$. We now make the following claim.
\begin{claim}\label{claim: count}
\[
\log_r |\mathrm{Ga}(P, G)|\leq \left(\frac12 - \xi^2\right)\left(\frac{n^2}{2} - \frac{k^2}{2}\right) + 3n^{23/12} + \log_r |\mathrm{Ga}(P, G')|.
\]
\end{claim}
\noindent\textit{Proof.} In the $i$-th iteration step of the above algorithm, if we add to $I$ a single vertex $v$, then the number of ways to color the incident edges of $v$ in $G_i$ satisfies
\[
\log_r\textstyle\prod_{\substack{e \text{ is incident to } \\ v \text{ in } G_i}}|P(e)|\leq d_{G_i}(v)\leq (\frac12 - \xi^2)(|G_i|-1).
\]
Now we assume that what we add is a pair of vertices $\{u, v\}$. 
For every $w\in N_{G_i}(u) \cap N_{G_i}(v)$, vertices $uvw$ either span a rainbow triangle in $P$, or satisfy $|P(uw)|=|P(vw)|=1$. Together with the fact that $uv$ is a typical $r$-edge, we obtain that the number of ways to color the edges, which are incident to $v$ or $u$ in $G_i$, satisfies
\begin{equation*}
\arraycolsep=1.4pt\def\arraystretch{1.2}
\begin{array}{ll}
\log_r \textstyle\prod_{\substack{e \text{ is incident to } \\ u \text{ or } v \text{ in } G_i}}|P(e)|&
\leq |G_i| - 2 -  |N_{G_i}(u) \cap N_{G_i}(v)| + 2n^{11/12} + 1 \\
&\leq (1 -2\xi^2)(|G_i| - 2) + 2n^{11/12} + 1.
\end{array}
\end{equation*}
From the above discussion, we conclude that the number of ways to color edges in $E(G) - E(G')$ satisfies 
\[
\log_r\textstyle\prod_{e \in E(G) - E(G')}|P(e)|
\leq \left(\frac12 - \xi^2\right)\left(\frac{n^2}{2} - \frac{k^2}{2}\right) + n( 1 +  2n^{11/12}),
\] which implies the claim.\qed\\

We now split the proof into several cases.\\

\noindent\textbf{Case 1: } $k\leq \xi^2 n$.\\
Then $|\mathrm{Ga}(P, G')|\leq r^{k^2/2}\leq r^{\xi^4n^2/2},$ and therefore by Claim~\ref{claim: count} and $\xi\ll 1$, we obtain that 
\[
\log_r |\mathrm{Ga}(P, G)|\leq \left(\frac12 - \xi^2\right)\frac{n^2}{2} + 3n^{23/12} + \xi^4n^2/2
\leq \frac{n^2}{4}  - \xi^2\frac{n^2}{4} + 3n^{23/12}.
\]
\\
\noindent\textbf{Case 2: } $e(G')>\left(\frac12 - 2\xi\right)k^2$ and $k> \xi^2 n$.\\
Since $2\xi\leq \log^{-11}n \leq  \log^{-11}k$, for sufficiently large $n$, Theorem~\ref{rgraph: high} indicates that $|\mathrm{Ga}(P, G')|\leq r^{k^2/4}$.
We claim that $k\leq (1 - \xi)n$, as otherwise we would have
\[
e(G)\geq e(G')> \left(\frac12 - 2\xi\right)k^2 >\left(\frac12 - 2\xi\right)(1 - \xi)^2 n^2\geq \left(\frac12 - 3\xi\right)n^2,
\]
which is contradiction with the assumption of the lemma. Therefore, by Claim~\ref{claim: count}, we obtain that
\[\arraycolsep=1.4pt\def\arraystretch{1.2}
\begin{array}{ll}
\log_r|\mathrm{Ga}(P, G)|&\leq \left(\frac12 - \xi^2\right)\left(\frac{n^2}{2} - \frac{k^2}{2}\right) + 3n^{23/12} +  \frac{k^2}{4}
\leq \frac{n^2}{4} - \xi^2\frac{n^2}{2} + \xi^2\frac{k^2}{2} + 3n^{23/12}\\
&\leq \frac{n^2}{4} - \xi^2\frac{n^2}{2} + \xi^2(1- \xi)^2\frac{n^2}{2} + 3n^{23/12}
\leq \frac{n^2}{4} - \xi^3\frac{n^2}{2} + 3n^{23/12}.
\end{array}
\]
\\
\noindent\textbf{Case 3: } $e(G')< \left(\frac14 + 2\xi\right)k^2$ and $k> \xi^2 n$.\\
Since $2\xi\ll 1$, Theorem~\ref{rgraph: low} indicates that $|\mathrm{Ga}(P, G')|\leq r^{k^2/4}$.
We claim that $k\leq (1 - \xi)n$, as otherwise we would have
\[\arraycolsep=1.4pt\def\arraystretch{1.2}
\begin{array}{ll}
e(G)&< \left(\frac{n^2}{2} - \frac{k^2}{2}\right) + \left(\frac14 + 2\xi\right)k^2 < \frac{n^2}{2} - \left(\frac14 - 2\xi\right)k^2\\
&<\frac{n^2}{2} - \left(\frac14 - 2\xi\right)(1 - \xi)^2n^2\leq \frac{n^2}{2} - \left(\frac14 - 3\xi\right)n^2=\left(\frac14 + 3\xi\right)n^2,
\end{array}
\]
which is contradiction with the assumption of the lemma. Similarly, as in Case 2, we obtain that
\[
\log_r |\mathrm{Ga}(P, G)|\leq \left(\frac12 - \xi^2\right)\left(\frac{n^2}{2} - \frac{k^2}{2}\right) + 3n^{23/12} + \frac{k^2}{4}
\leq \frac{n^2}{4} - \xi^3\frac{n^2}{2} + 3n^{23/12}.
\]
\\
\noindent\textbf{Case 4: } $(\frac14 + 2\xi)k^2\leq e(G')\leq (\frac12 - 2\xi)k^2$ and $k> \xi^2 n$.\\
Denote by $e_r(G')$ the number of $r$-edges of $P$ in $G'$. Let $A = \{v\in V(G') \mid d_{G'}(v)\leq \left(\frac12 + \xi\right)k\}$. 
\begin{claim}\label{claim: setA}
All the typical $r$-edges of $G'$ have both endpoints in $A$.
\end{claim}
\noindent\textit{Proof.}
First, by the construction of $G'$, we have the following two properties: for  every $v\in V(G')$,
\begin{equation}\label{eq: mindegree}
d_{G'}(v)> \left( \frac12 - \xi^2\right)(k -1),
\end{equation}
and for every typical $r$-edge $uv$ in  $G'$, 
\begin{equation}\label{eq: sumdegree}
d_{G'}(u) + d_{G'}(v) \leq 2 + (k-2) + |N_{G_i}(u) \cap N_{G_i}(v)|< (1 + 2\xi^2)k.
\end{equation}
Suppose that there exists a typical $r$-edge $uv$ such that $u$ is not in $A$, i.e.~$d_{G'}(u)> \left(\frac12 + \xi\right)k$. Then by~(\ref{eq: mindegree}) and $\xi \ll 1$, we  have 
\[
d_{G'}(u) + d_{G'}(v) >\left(\frac12 + \xi\right)k +  \left( \frac12 - \xi^2\right)(k -1)>( 1 + 2\xi^2)k,
\]
which contradicts~(\ref{eq: sumdegree}).\qed\\

\noindent\textbf{Subcase 4.1: } $|A|\leq \xi k$.\\
By Proposition~\ref{prop: redge} and Claim~\ref{claim: setA}, we have 
\[
e_r(G') \leq \binom{|A|}{2} + n^{11/6} \leq \xi^2\frac{k^2}{2} + n^{11/6}.
 \]
Therefore, together with the assumption of Case 4, we obtain that
\[\arraycolsep=1.4pt\def\arraystretch{1.2}
\begin{array}{ll}
\log_r |\mathrm{Ga}(P, G')| & \leq \log_r \left(r^{e_r(G')}2^{e(G') - e_r(G')}\right) \leq \frac12\left(e(G') + e_r(G')\right)\\
& \leq \frac12\left(\left(\frac12 - 2\xi\right)k^2 +  \xi^2\frac{k^2}{2} + n^{11/6}\right)
=\frac{k^2}{4}- \left(\xi - \frac{1}{4}\xi^2\right)k^2  + \frac12 n^{11/6}.
\end{array}
\]
Then by Claim~\ref{claim: count},
\[\arraycolsep=1.4pt\def\arraystretch{1.2}
\begin{array}{ll}
\log_r |\mathrm{Ga}(P, G)|
&\leq \left(\frac12 - \xi^2\right)\left(\frac{n^2}{2} - \frac{k^2}{2}\right) + 3n^{23/12} + \frac{k^2}{4}- \left(\xi - \frac{1}{4}\xi^2\right)k^2  + \frac12 n^{11/6}\\
&\leq  \frac{n^2}{4} - \xi^2\frac{n^2}{2} - \left(\xi - \frac{3}{4}\xi^2\right)k^2 +  4n^{23/12}
\leq \frac{n^2}{4} - \xi^2\frac{n^2}{2}  +  4n^{23/12},
\end{array}
\]
where the last inequality is given by $\xi \ll 1$.\\

\noindent\textbf{Subcase 4.2: } $|A|> \xi k$.\\
By the definition of $A$, the number of non-edges of $G'$ is at least
\begin{equation}\label{eq: nonedge}
\frac12\left(k-1- \left(\frac12 + \xi\right)k\right)|A|
=\frac12\left( \left(\frac12 - \xi\right )k - 1\right)|A|.
\end{equation}
We first claim that 
\begin{equation}\label{eq: sizeA}
|A| \leq \frac{1 - 8\xi}{1 - 2\xi} k,
\end{equation}
 as otherwise we would obtain that the number of non-edges of $G'$ is more than
\[
\frac12\left(\frac12 - \xi\right )k \cdot \frac{1 - 8\xi}{1 - 2\xi}k - \frac{|A|}{2}\geq \left(\frac14 - 2\xi\right)k^2 - \frac{k}{2}
\]
which contradicts the assumption of Case 4. Inequality~(\ref{eq: sizeA}) implies that
\begin{equation}\label{eq: ineq}
(1 -2 \xi)k - |A| \geq 4\xi k.
\end{equation}
By Propositions~\ref{prop: edge} and~\ref{prop: redge}, since $|A| > \xi k > \xi^3 n$, we have 
\begin{equation*}
e_r(G') \leq e(G'[A]) + n^{11/6} \leq \frac{|A|^2}{4} + \frac{e^4}{6n^{1/3}\xi^9}|A|^2 + n^{11/6},
\end{equation*}
as otherwise we would find more than $n^{-1/3}\binom{n}{3}$ rainbow triangles, which contradicts the assumption that $P$ is a Gallai $r$-template of $G$.
Since $\xi\gg n^{-1/33}$, we have 
\begin{equation}\label{eq: redge}
e_r(G') \leq \frac{|A|^2}{4} + \frac{\xi^2}{2}|A|^2 + n^{11/6}.
\end{equation}
Combining (\ref{eq: nonedge}), (\ref{eq: ineq}) and (\ref{eq: redge}), we have 
\[\arraycolsep=1.4pt\def\arraystretch{1.2}
\begin{array}{ll}
\log_r |\textrm{Ga}(P, G')|& \leq \frac12 \left( e(G') + e_r(G') \right) 
\leq \frac12 \left( \binom{k}{2} -  \frac12\left( \left(\frac12 - \xi\right )k - 1\right)|A| + \frac{|A|^2}{4} + \frac{\xi^2}{2}|A|^2 + n^{11/6}\right)\\
& \leq \frac{k^2}{4} - \frac{|A|}{8}\left( (1 - 2\xi)k - |A|\right) + \frac{\xi^2}{4}|A|^2 + \frac12 n^{11/6}
\leq \frac{k^2}{4} - \frac{\xi}{2}|A|k + \frac{\xi^2}{4}|A|^2+ \frac12 n^{11/6}.
\end{array}
\]
Then by Claim~\ref{claim: count} and the assumption of Subcase 4.2, we obtain that
\[\arraycolsep=1.4pt\def\arraystretch{1.2}
\begin{array}{ll}
\log_r |\mathrm{Ga}(P, G)|
&\leq \left(\frac12 - \xi^2\right)\left(\frac{n^2}{2} - \frac{k^2}{2}\right) + 3n^{23/12} + \frac{k^2}{4} - \frac{\xi}{2}|A| k + \frac{\xi^2}{4}|A|^2 + \frac12 n^{11/6}\\
&< \left(\frac12 - \xi^2\right)\left(\frac{n^2}{2} - \frac{k^2}{2}\right) + 3n^{23/12} + \frac{k^2}{4} - \frac{\xi^2}{2}k^2 + \frac{\xi^2}{4}n^2 + \frac12 n^{11/6}\\
& \leq \frac{n^2}{4} - \xi^2 \frac{n^2}{4} + 4n^{23/12}.
\end{array}
\]

\end{proof}

\noindent\textit{Proof of Theorem~\ref{rgraph: middle}. }
Let $\mathcal{C}$ be the collection of containers given by Theorem \ref{container}. Theorem \ref{container} indicates that every Gallai $r$-coloring of $G$ is a subtemplate of some $P\in \mathcal{C}$ and $|\mathcal{C}|\leq 2^{cn^{-1/3}\log^2 n\binom{n}{2}}$ for some constant $c$, which only depends on $r$.  
We may assume that all templates $P$ in $\mathcal{C}$ are Gallai $r$-templates of $G$. By Property (ii) of Theorem~\ref{container}, we always have $\mathrm{RT}(P)\leq n^{-1/3}\binom{n}{3}$. Suppose that for a template $P$ there exists an edge $e\in E(G)$ with $|P(e)|=0$. Then we would obtain $|\mathrm{Ga}(P, G)| = 0$ as a Gallai $r$-coloring of $G$ requires at least one color on each edge. 
Now applying Lemma~\ref{lemma: middle} on every container $P\in \mathcal{C}$, we obtain that the number of Gallai $r$-colorings of $G$ is at most
\[
\sum_{P\in \mathcal{C}}|\mathrm{Ga}(P, G)|\leq |\mathcal{C}| \cdot r^{\frac{n^2}{4} - \xi^3\frac{n^2}{2} + 4n^{23/12}}< r^{\lfloor n^2/4 \rfloor},
\]
where the last inequality follows from $\xi\gg n^{-1/36}$ for $n$ sufficiently large.

\bibliographystyle{siam}

\end{document}